\documentclass{nm}
\renewcommand\thisnumber{x}
\renewcommand\thisyear {200x}

\renewcommand\thisvolume{xx}

\usepackage{charter}
\usepackage[charter]{mathdesign}
 \usepackage{graphicx}

\begin{document}


\markboth{Z.Z. Zhang,  W.H. Deng, and H.T. Fan}{Finite difference schemes for the tempered fractional Laplacian}
\title{Finite difference schemes for the tempered fractional Laplacian}



\author[Zhijiang Zhang, Weihua Deng, and HongTao Fan]{Zhijiang Zhang\affil{1}, Weihua Deng\affil{1}\comma\corrauth, HongTao Fan\affil{2}}
\address{\affilnum{1}\ School of Mathematics and Statistics, Gansu Key Laboratory of Applied Mathematics and Complex Systems, Lanzhou
University, Lanzhou 730000, P.R. China\\
\affilnum{2}\ College of Science, Northwest A{\&}F University, Yangling 712100, Shaanxi, P.R. China}

\emails{{\tt dengwh@lzu.edu.cn} (Weihua Deng)}

\begin{abstract}
The second and all higher order moments of the $\beta$-stable L\'{e}vy process diverge, the  feature of which is sometimes referred to as shortcoming of the model when applied to physical processes. So, a parameter $\lambda$ is introduced to exponentially temper the L\'{e}vy process. The generator of the new process is tempered fractional Laplacian $(\Delta+\lambda)^{\beta/2}$
[W.H. Deng, B.Y. Li, W.Y. Tian, and P.W. Zhang, Multiscale Model. Simul., in press, 2017].
%
 In this paper, we first design the finite difference schemes for the tempered fractional Laplacian equation with the generalized Dirichlet type boundary condition, their accuracy depending on the regularity of the exact solution on $\bar{\Omega}$. Then the techniques of effectively solving the resulting algebraic equation are presented, and the performances of the schemes are demonstrated by several numerical examples.

\end{abstract}



\maketitle

\section{Introduction}\label{sec1}

The fractional Laplacian $\Delta^{\beta/2}$ is the generator of the $\beta$-stable L\'{e}vy process, in which the random displacements executed by jumpers are able to walk to neighboring or nearby sites, and also perform excursions to remote sites by way of L\'{e}vy flights \cite{Pozrikidis:16,Applebaum:09,Peszat:07}. The distribution of the jump length of $\beta$-stable L\'{e}vy process obeys the isotropic power-law measure $|x|^{-n-\beta}$, where $n$ is the dimension of the space. The extremely long jumps of the process make its second and higher order moments divergent, sometimes being referred to as a shortcoming when it is applied to physical model in which one expects regular behavior of moments \cite{Zaburdaev:15}.
The natural idea to damp the extremely long jumps is to introduce a small damping parameter $\lambda$ to the distribution of jump lengths, i.e., $e^{-\lambda |x|}|x|^{-n-\beta}$. With small $\lambda$, for short time, it displays the dynamics of L\'{e}vy process, while for sufficiently long time the dynamics will transit slowly from superdiffusion to normal diffusion. The generator of the tempered L\'{e}vy process is the tempered fractional Laplacian $(\Delta+\lambda)^{\beta/2}$ \cite{Deng:17}. The tempered fractional Laplacian equation governs the probability distribution function of the position of the particles.


This paper focuses on developing the finite difference schemes for the tempered fractional Laplacian  equation
\begin{eqnarray}\label{Laplacedirichelt}
\left\{\begin{array}{ll}
-(\Delta +\lambda)^ {\beta/2}u(x)=f(x),&~~ x\in \Omega,\\
u(x)=g(x),&~~x\in \mathbb{R}\backslash\Omega,
\end{array}\right.
\end{eqnarray}
where  $\beta\in (0,2),\,\lambda\ge0,\Omega=(a,b) $,  and
\begin{eqnarray}\label{temperedLaplacianoperator}
(\Delta +\lambda)^{\beta/2}u(x):=
-c_{\beta} {~\rm P.V.~} \int_{\mathbb{R}}\frac{u(x)-u(y)}{e^{\lambda|x-y|}|x-y|^{1+\beta}}dy
\end{eqnarray}
with
\begin{eqnarray}
c_{\beta}=\left\{\begin{array}{ll}
\frac{\beta\Gamma(\frac{1+\beta}{2})}{2^{1-\beta}\pi^{1/2}\Gamma(1-\beta/2)}& ~{\rm for~}\lambda=0
 ~{\rm or}~ \beta=1,\\[3pt] \frac{\Gamma(\frac{1}{2})}{2\pi^{\frac{1}{2}}|\Gamma(-\beta)|} &~{\rm for}~\lambda>0~{\rm and}~\beta\not=1,
\end{array}\right.
\end{eqnarray}
and ${~\rm P.V.~}$ being the limit of the integral over
$\mathbb{R}\backslash B_{\epsilon} (x)$ as $\epsilon\to 0$. The tempered operator in
(\ref{temperedLaplacianoperator}) is the generator of the
 tempered  symmetric $\beta$-stable L\'{e}vy process $x(t)$ determined by the L\'{e}vy-Khintchine representation
\begin{eqnarray}
{\bf{E}}(e^{i\omega x})=\exp\left(t\int_{\mathbb{R}\backslash\{0\}} \left(e^{i\omega y}-1-i\omega y \chi_{\{|y|<1\}}\right)   \nu (dy) \right).
\end{eqnarray}
Here ${\bf{E}}$ is the expectation,
$ i=\sqrt{-1},\,\chi_{S}:=\left\{\begin{array}{ll} 1,&x\in S,\\0,&s\notin S\end{array}\right.$,
and $\nu (dy)=e^{-\lambda|y|}|y|^{-\beta-1}dy $ is the tempered L$\acute{\rm e}$vy  measure. The parameter $\lambda \,(>0)$ fixes the decay rate of big jumps, while $\beta$ determines
the relative importance of smaller jumps in the path of the process.
Model (\ref{Laplacedirichelt}) corresponds to the one-dimensional case of the initial and boundary value problem in Eq. (49) recently proposed in \cite{Deng:17}, and the existence and uniqueness of its weak solution have been shown in \cite{Zhang:17}.
Obviously, when $\lambda=0$, (\ref{temperedLaplacianoperator}) reduces to the fractional Laplacian \cite{Pozrikidis:16}
\begin{eqnarray} \label{eqnarrddddd}
(\Delta)^{\beta/2}u(x):=-c_{\beta} {~\rm P.V.~} \int_{\mathbb{R}}\frac{u(x)-u(y)}{|x-y|^{1+\beta}}dy.
\end{eqnarray}
It is well known that for the proper classes of functions that decay quickly enough at infinity,
the fractional Laplacian can be rewritten as the combination of the left and right Riemann-Liouville fractional derivatives ${}_{-\infty}D_x^{\beta}u(x)$ and ${}_xD_{\infty}^{\beta}u(x)$
(the so-called Riesz fractional derivative)\cite{Yang:10}, i.e.,
\begin{eqnarray} \label{fractionaLaplacian}
-\Delta^{\beta/2}u(x)=\frac{{}_{-\infty}D_x^{\beta}u(x)+{}_xD_{\infty}^{\beta}u(x)}{2\cos(\beta\pi/2)},~~~~\beta\not =1.
\end{eqnarray}
The similar result also holds for the tempered fractional Laplacian. In fact, letting $u(x)\in H^{\beta}(\mathbb{R})$ and $\mathscr{F}[u(x)](\omega):=\int_{\mathbb{R}}u(x)e^{-ix\omega}dx$ be its Fourier transform, we have \cite[Propositions 2.1 and 2.2]{Zhang:17}
 \begin{eqnarray}\label{lemma1eq1}
 &&\mathscr{F}\left[(\Delta +\lambda\right)^{\beta/2}u(x)]({\omega})\nonumber\\
 &&~=(-1)^{\lfloor \beta \rfloor}\left(\lambda^\beta-\left(\lambda^2+\left|{\omega}\right|^2\right)^{\frac{\beta}{2}}
 \cos\left(\beta \arctan\left(\frac{\left|{\omega}\right|}{\lambda}\right)\right)\right)\mathscr{F}[u]({\omega}),
\end{eqnarray}
 where $\lfloor\beta\rfloor:=\left\{ z\in \mathbb{N}: 0\le\beta-z<1\right\}$. Note that
\begin{eqnarray}
\frac{1}{2}\left(\lambda^2+\left|{\omega}\right|^2\right)^{\frac{\beta}{2}}\cos\left(\beta \arctan\left(\frac{\left|{\omega}\right|}{\lambda}\right)\right)
=(\lambda +i\omega)^{\beta}+(\lambda -i\omega)^{\beta}.
\end{eqnarray}
 A simple calculation  yields that
\begin{eqnarray}\label{eqnarrddddd2}
-\left(\Delta+\lambda\right)^{\beta/2}u(x)=(-1)^{\lfloor\beta \rfloor}\, \frac{\partial _{x} ^{\beta,\lambda}u(x)+\partial_{-x}^{\alpha, \lambda} u(x)}{2},~~~\beta\not=1,
\end{eqnarray}
where
\begin{eqnarray}
\partial _{x} ^{\beta,\lambda}u(x):=
\left\{\begin{array}{ll}
\mathscr{F}^{-1}\left[\left((\lambda+i\omega)^{\beta}-\lambda^{\beta}\right)\hat{u}(\omega)\right](x),&\beta\in(0,1),\\[5pt]
\mathscr{F}^{-1}\left[\left((\lambda+i\omega)^{\beta}-i\omega {\beta}\lambda^{{\beta}-1}-\lambda^{\beta}\right)\hat{u}(\omega)\right](x),&\beta\in(1,2)
\end{array}\right.
\end{eqnarray}
and
 \begin{eqnarray}
\partial_{-x}^{\alpha, \lambda} u(x):=
\left\{\begin{array}{ll}
\mathscr{F}^{-1}\left[\left((\lambda-i\omega)^{\beta}-\lambda^{\beta}\right)\hat{u}(\omega)\right](x),&\beta\in(0,1),\\[5pt]
\mathscr{F}^{-1}\left[\left((\lambda-i\omega)^{\beta}+i\omega {\beta}\lambda^{{\beta}-1}-\lambda^{\beta}\right)\hat{u}(\omega)\right](x),&\beta\in(1,2)
\end{array}\right.
\end{eqnarray}
are  the left and right  normalized tempered  Riemann-Liouville fractional operations being given in
  \cite{Sabzikar:15,Baeumer:10}, and their representations  in real space can be founded in   \cite{Lican:15}.

Nowadays, many finite difference schemes have been proposed to solve equations with the Riemann-Liouville type fractional derivatives
in (\ref{fractionaLaplacian}) or (\ref{eqnarrddddd2}) under zero boundary conditions \cite{Baeumer:10,Lican:15,Meerschaert:04,Sabzikar:15,Zhao:14,Zhao:16},
 which usually are constructed based on the Gr{\"u}nwal formula or its variants; and there are also a lot of discussions on time-fractional operators or other numerical methods, e.g., \cite{HuangJF:14, Sheng:17}. To the best of our knowledge, it seems that very few numerical schemes are based on the singular integral definition (\ref{eqnarrddddd}) to approximate the fractional Laplacian. In \cite{Gao:14}, to study the mean exit time and  escape probability of the dynamical systems driven by non-Gaussian L\'{e}vy noises,
 the fractional Laplacian is  approximated numerically by   a ``punched-hole'' trapezoidal rule.
 The finite difference and finite element methods with systemically theoretical analysis for solving model (\ref{Laplacedirichelt}) with $\lambda=0$
  are presented in \cite{Huang:14} and \cite{Acosta:171}, respectively. Usually, even for problems  with $g(x)=0$, to preform  the convergence analysis,  the finite element methods
   refer to  the regularity of the exact solution on $\Omega$ \cite{Acosta:171,Zhang:17}  while the finite difference methods require the regularity
  on the whole line \cite{Baeumer:10,Meerschaert:04,Lican:15,Sabzikar:15,Huang:14}. The finite difference schemes provided in this paper for the tempered fractional Laplacian equation (\ref{Laplacedirichelt}) just depend on the regularity of $u(x)$ on $\bar{\Omega}$. We give the detailedly theoretical analysis and effective algorithm implementation.

The rest of this paper is organized as follows. In Section \ref{sec2} we discuss the numerical
 schemes of (\ref{Laplacedirichelt}) with $\beta\in (0,1)$. We first derive the finite difference discretizations of the
 tempered fractional Laplacian based on the singular integral definition (\ref{temperedLaplacianoperator}), and then
 give convergence analysis and the related implementation techniques for  solving the resulting algebraic equation with preconditioning.
 Two types of preconditioners are considered. In Section \ref{sec3}, we extend the suggested finite difference schemes to the case $\beta\in [1,2)$, and
 most of the results and implementation techniques still hold.
 Numerical simulations are presented in Section \ref{sec4} and we conclude the paper in Section \ref{sec5}.

Throughout the paper by the notation $A\lesssim B$ we mean that $A$ can be bounded by a multiple of $B$,
independent of the parameters they may depend on, while the expression $A\simeq B$ means that $A\lesssim B\lesssim A$.
We also use $C$ to denote a constant, which may be different for different lines.

\section{Finite difference scheme for the case $\beta\in (0,1)$}\label{sec2}
In this section, we discuss the finite difference scheme for the model
(\ref{Laplacedirichelt}) with $\beta\in (0,1)$.
\subsection{Derivation of the scheme }\label{subsec21}
Let $\Omega=(a,b)$ and $s, s_1\in \{0,1\}$. We partition $\Omega$ uniformly  into $a=x_0<x_1<\cdots <x_{M}<x_{M+1}=b$
with $x_i=a+ih, \,h=\frac{b-a}{M+1}$. Then  for $i=1,2,\cdots,M$, we have
\begin{eqnarray}
&&\int_{-\infty}^{\infty} \frac{u(x_i)-u(y)}{e^{\lambda|x_i-y|}\left|x_i-y\right|^{1+\beta}}dy\nonumber\\
&&~=\int_a^{x_{i-1}}g_1(i,s,y)(x_i-y)^{s-1-\beta}dy+\int_{x_{i+1}}^b g_2(i,s,y)(x_i-y)^{s-1-\beta} dy \nonumber\\
&&~~+\int_0^hg_3(i,s_1,y)y^{s_1-1-\beta}dy+\left(B_1(i)+B_2(i)\right)u(x_i)-d_1(i)-d_2(i),
\end{eqnarray}
where
\begin{eqnarray}
&&g_1(i,s,y):=\frac{\left({u(x_i)-u(y)}\right)}{(x_i-y)^{s}}{e^{\lambda(y-x_i)}},~~y\in [a,x_{i-1}],\label{bianliangdaihua1}\\
&& g_2(i,s,y):=\frac{\left({u(x_i)-u(y)}\right)}{(y-x_i)^{s}}{e^{\lambda(x_i-y)}},~~y\in [x_{i+1},b],\\
 &&g_3(i,s,y):=\frac{\left(2u(x_i)-u(x_i-y)-u(x_i+y)\right)}{y^{s_1}}e^{-\lambda y},~~y\in [0,h],\label{bianliangdaihua2}
 \end{eqnarray}
 and
 \begin{eqnarray}
&&B_1(i)=\int_{-\infty}^a\frac{e^{\lambda(y-x_i)}}{(x_i-y)^{1+\beta}}dy,~~~~ B_2(i)=\int_b^{\infty}\frac{e^{\lambda(x_i-y)}}{(y-x_i)^{1+\beta}}dy,\\
&&d_1(i)=\int_{-\infty}^a \frac{g(y)}{e^{\lambda(x_i-y)}(x_i-y)^{1+\beta}}dy, ~~d_2(i)=\int_b^{\infty} \frac {g(y)}{e^{\lambda(y-x_i)}(y-x_i)^{1+\beta}}dy.~~~~
\end{eqnarray}
Define an interpolation operator by
\begin{eqnarray} \label{edde33sseee}
I_{\left[a_1,a_2\right]}\left(p(y)\right):=\frac{y-a_1}{h}p(a_2)+\frac{a_2-y}{h}p(a_1),~~y\in [a_1,a_2].
\end{eqnarray}
We approximate the term
 $ \int_a^{x_{i-1}}g_1(i,y)(x_i-y)^{s-1-\beta}dy$
  by
\begin{eqnarray}\label{discretizationcoeff1}
&&\mathcal{A}_1(s,g_1)=\sum_{k=1}^{i-1}\int_{x_{k-1}}^{x_{k}}I_{\left[x_{k-1},\,x_k\right]}\left(g_1(i,s,y)\right)\,(x_i-y)^{s-1-\beta}dy,~~~~~~\nonumber\\
&&~~~~~~~~~~~~~~~=\sum_{k=1}^{i-1}h^s\left(g_1(i,s,x_{k-1})A_1(i,s,k)+g_1(i,s,x_k)A_2(i,s,k)\right)
\end{eqnarray}
with
\begin{eqnarray}
&&A_1(i,s,k)=\frac{1}{h^{s+1}}\int_{x_{k-1}}^{x_k}(x_k-y)(x_i-y)^{s-1-\beta}dy>0,\label{coeffieedddeddee1}\\
&&A_2(i,s,k)=\frac{1}{h^{s+1}}\int_{x_{k-1}}^{x_k}(y-x_{k-1})(x_i-y)^{s-1-\beta} dy>0;\label{coeffieedddeddee2}
\end{eqnarray}
the term
$\int_{x_{i+1}}^b g_2(i,y)(y-x_i)^{s-1-\beta} dy$ is approximated
 by
\begin{eqnarray}
&&\mathcal{A}_2(s,g_2)= \sum_{k=i+2}^{M+1}\int_{x_{k-1}}^{x_k}I_{\left[x_{k-1},\,x_k\right]}\left(g_2(i,s,y)\right)\,(y-x_i)^{s-1-\beta} dy~~~~~~~~\nonumber\\
&&~~~~~~~~~~~~~~~=\sum_{k=i+2}^{M+1}h^s\left(g_2(i,s,x_{k-1})A_3(i,s,k)+g_2(i,s,x_k)A_4(i,s,k)\right)~~~~~~~~~~~\label{discretizationcoeff2}
\end{eqnarray}
with
\begin{eqnarray}
&&A_3(i,s,k)=\frac{1}{h^{s+1}}\int_{x_{k-1}}^{x_k} (x_k-y)(y-x_i)^{s-1-\beta}dy>0,\label{coeffieedddeddee3}\\
&&A_4(i,s,k)=\frac{1}{h^{s+1}} \int_{x_{k-1}}^{x_k}\left(y-x_{k-1}\right) (y-x_i)^{s-1-\beta} dy>0;\label{coeffieedddeddee4}
\end{eqnarray}
 and the term
$ \int_0^h g_3(i,s_1,y) y^{s_1-1-\beta} dy $
 by
\begin{eqnarray}
\mathcal{A}_3(s_1,g_3)=\int_0^h I_{\left[0,\,h\right]}\left(g_3(i,s_1,y)\right)\,y^{s_1-1-\beta}dy=\frac{h^{s_1-\beta}}{s_1-\beta+1} g_3(i,s_1,h). \label{discretizationcoeff3}
\end{eqnarray}

Assume $g_1(i,s,y)\in C^2[a,x_{i-1}],g_2(i,s,y)\in C^2[x_{i+1},b],$ and $g_3(i,s_1,y)\in C^2[0,h]$.
By the Lagrange interpolation error remainder, we have
\begin{eqnarray}
&&\left|\int_a^{x_{i-1}}g_1(i,s,y)(x_i-y)^{s-1-\beta}dy-\mathcal{A}_1(s,g_1)\right|\nonumber\\
&&~\le \frac{\big\|g^{(2)}_1(i,s,y)\big\|_{L^\infty[a,\,x_{i-1}]} h^2}{2\left|s-\beta\right|}\left|h^{s-\beta}-(x_i-a)^{s-\beta}\right|,\label{localtruncation1}\\
&&\left|\int_{x_{i+1}}^b g_2(i,s,y)(y-x_i)^{s-1-\beta} dy-\mathcal{A}_2(s,g_2)\right|\nonumber\\
&&~~\le \frac{\big\|g^{(2)}_2(i,s,y)\big\|_{L^\infty[x_{i+1},\,b]} h^2}{2\left|s-\beta\right|}\left|h^{s-\beta}-(b-x_i)^{s-\beta}\right|, \label{localtruncation2}
\end{eqnarray}
and
\begin{eqnarray}
&&\left|\int_0^h g_3(i,s_1,y) y^{s_1-1-\beta} dy -\mathcal{A}_3(s_1,g_3)\right|\nonumber\\
&&~= \left|\int_0^h\frac{g_3^{(2)}(i,s_1,y)}{2} y(h-y)y^{s_1-1-\beta} dy\right|\nonumber\\
&&~\le \frac{\big\| g_3^{(2)}(i,s_1,y)\big\|_{L^{\infty}([0,h])}h^{s_1-\beta+2}}{2(s_1-\beta+1)(s_1-\beta+2)}.\label{localtruncation3}
\end{eqnarray}
If $u(y)\in C^2(\bar{\Omega})$, it can be noted that
\[\big\|g^{(2)}_1(i,0,y)\big\|_{L^{\infty}[a,x_{i-1}]}, \big\|g_2^{(2)}(i,0,y)\big\|_{L^{\infty}[x_{i+1},b]}, ~{\rm and~} \big\|g_3^{(2)}
(i,0,y)\big\|_{L^{\infty}[0,h]}\]
  are bounded, and the bounds may depend on the values of  $u^{(k)}(y), k=0,1,2$ on $\bar{\Omega}$, but  independent of $h$;
and if $u\in C^{(3)}(\bar{\Omega})$, letting $h(x):=\frac{u(x_i)-u(x)}{x_i-x}$ and using  Taylor's expansion,
\begin{eqnarray}
\left\|h(x)\right\|_{L^\infty[a,x_{i-1}]},\,\big\|h^{(1)}(x)\big\|_{L^\infty[a,x_{i-1}]},\,{\rm and}~\big\|h^{(2)}(x)\big\|_{L^\infty[a,x_{i-1}]}
\end{eqnarray}
also are bounded, and the bounds may depend on the values of  $u^{(k)}(x), k=0,1,2,3$ on $\bar{\Omega}$, but also independent of $h$.  Thus
\begin{eqnarray}
\big\|g^{(2)}_1(i,1,y)\big\|_{L^\infty[a,x_{i-1}]}=\big\|\left(\lambda^2h(x)+2\lambda h^{(1)}(x)+h^{(2)}(x)\right)e^{\lambda(y-x_i)}\big\|_{L^\infty[a,x_{i-1}]}\le C.~~
\end{eqnarray}
Similarly, we have
\begin{eqnarray}\label{errorestimatematrix1}
\big\|g_2^{(2)}(i,1,y)\big\|_{L^{\infty}[x_{i+1},b]}\le C,~~\big\|g_3^{(2)}(i,1,y)\big\|_{L^{\infty}[0,h]}\le C.
\end{eqnarray}
Therefore, combining  (\ref{discretizationcoeff1}), (\ref{discretizationcoeff2}), (\ref{discretizationcoeff3}),
and (\ref{localtruncation1})-(\ref{errorestimatematrix1}),  for $\beta\in (0,1)$ and $i=1,2,\cdots, M$, it follows that
\begin{equation}
\begin{array}{lll}
&&-(\Delta +\lambda)^ {\beta/2}u(x_i)= \mathcal{A}_1(s,g_1)+\mathcal{A}_2(s,g_2)+\mathcal{A}_3(s_1,g_3)\nonumber\\
&&~~~~~~~~~~~~~~~~~~~~~~~~~~~~~~+(B_1(i)+B_2(i))u(x_i)-(d_1(i)+d_2(i))+r_h^i,\label{discretzaitionequation}
\end{array}
\end{equation}
 where
 \begin{eqnarray}\label{localerror1}
 \left|r_h^i\right|=\left\{\begin{array}{l}
 \mathcal{O}(h^{2-\beta}) {~~~~\rm for~}  u\in C^2(\bar{\Omega}), \,s=s_1=0, \\
  \mathcal{O}(h^{2}){~~~~~~~~\rm  for~} u\in C^3(\bar{\Omega}), \,s=s_1=1.
  \end{array}\right.
 \end{eqnarray}
 \begin{remark}
For $\lambda=0$, the bounds of $ \big\|g^{(2)}_1(i,s,y)\big\|_{L^{\infty}[a,x_{i-1}]}, \big\|g_2^{(2)}(i,s,y)\big\|_{L^{\infty}[x_{i+1},b]}$ and
$\big\|g_3^{(2)}(i,s_1,y)\big\|_{L^{\infty}[0,h]}$ only depend on the values of $u^{(2)}(x)$ on $\bar{\Omega}$ for $s=s_1=0$ and the values of $u^{(3)}(x)$ on $\bar{\Omega}$ for $s=s_1=1$.
 \end{remark}
 Define
\begin{eqnarray}
&&F_1:=f(x_1)+d_1(1)+d_2(1)+\frac{h^{-\beta} e^{-\lambda h}}{s_1+1-\beta}u(a)+A_4(1,s,M+1) \frac{e^{-\lambda Mh}}{M^{s}} u(b),~~~~~~~~~~~\\
&&F_M:=f(x_M)+d_1(M)+d_2(M)+\frac{h^{-\beta}e^{-\lambda h}}{s_1-\beta+1} u(b)+A_1(M,s,1)\frac{e^{-\lambda M h}}{M^s}u(a),
\end{eqnarray}
\begin{eqnarray}
&&F_i:=f(x_i)+d_1(i)+d_2(i)+A_1(i,s,1)\frac{e^{-\lambda ih}}{i^s}u(a)~~~~~~~~~~\nonumber\\
&&~~~~~~~~+A_4(i,s,M+1)\frac{e^{-\lambda(M+1-i) h}}{(M+1-i)^{s}} u(b), ~~i=2,3,\cdots,M-1,
\end{eqnarray}
and
\begin{eqnarray}\label{stiffenessmatrix1}
h_{i,j} :=\left\{\begin{array}{ll}
-\left(A_1(i,s,j+1)+A_2(i,s,j)\right)\frac{e^{-\lambda (i-j)h}}{(i-j)^{s}}, & 1\le j\le i-2,\\[3.5pt]
-\frac{h^{-\beta}e^{-\lambda h}}{s_1+1-\beta}-A_2(i,s,i-1)e^{-\lambda h},& j=i-1,\\[3.5pt]
-\frac{h^{-\beta}e^{-\lambda h}}{s_1+1-\beta}-A_3(i,s,i+2)e^{-\lambda h},& j=i+1,\\[3.5pt]
-\left(A_3(i,s,j+1)+A_4(i,s,j)\right)\frac{e^{-\lambda (j-i)h}}{(j-i)^{s}}, &  i+2\le j\le M,
\end{array}\right.
\end{eqnarray}
with the  $h_{i,i}$ satisfying
\begin{eqnarray}
&&h_{i,i}+\sum_{j=1,j\not= i}^M h_{i,j}-B_1(i)-B_2(i)\nonumber\\[3pt]
&&~=\left\{ \begin{array}{ll}
\frac{h^{-\beta}e^{-\lambda h}}{s_1-\beta+1}+A_4(i,s,M+1)\frac{e^{-\lambda (M+1-i) h}}{(M+1-i)^{s}},& i=1,\\[3.5pt]
A_1(i,s,1)\frac{e^{-\lambda i h}}{i^s}+A_4(i,s,M+1)\frac{e^{-\lambda (M+1-i) h}}{(M+1-i)^{s}},& 2\le i\le M-1,\\[3.5pt]
\frac{h^{-\beta}e^{-\lambda h}}{s_1-\beta+1}+A_1(i,s,1)\frac{e^{-\lambda i h}}{i^s},& i=M.
\end{array} \right.~~~~~~~~\label{stiffenessmatrix222}
\end{eqnarray}
Let
 \begin{eqnarray}
 &&{\bf H}:=\left(h_{i,j}\right)_{i,j=1}^M,~~~ U:=\left(u(x_1),u(x_2),\cdots, u(x_M)\right)^{\rm T},\nonumber\\
 &&F:=\left(F_1,F_2,\cdots, F_M\right)^{\rm T},~~R_h:=\left(r_h^1,r_h^2,\cdots, r_h^M\right)^{\rm T}.
 \end{eqnarray}
 Then making use of (\ref{discretzaitionequation}), it holds that
 \begin{eqnarray}\label{exactequation}
 {\bf H}U=F+R_h.
 \end{eqnarray}
 We denote the numerical solution of $u$ at $x_i$ as  $U_i$ and define $U_h:=\left(U_1,U_2,\cdots, U_M\right)^{\rm T}$. By discarding the
 truncation error $R_h$ and replacing $U$ by $U_h$,  the numerical scheme of (\ref{Laplacedirichelt}) can be given as
 \begin{eqnarray}\label{matrixform}
 {\bf H}U_h=F.
 \end{eqnarray}

\subsection{Error Estimates}
By a simple calculation, for $\beta\in (0,1)$,   we have
\begin{eqnarray}\label{stiffenessmatrix2}
\begin{array}{l}
A_1(i,s,j+1)+A_2(i,s,j)=A_3(i,s,j+1)+A_4(i,s,j)\\[4.0pt]
~=C_{\beta,s}\left((2|i-j|^{1-\beta+s}-(|i-j|-1)^{1-\beta+s}-(|i-j|+1)^{1-\beta+s}\right),\\[4.0pt]
A_2(i,s,i-1)=C_{\beta,s}\left(2-\beta+s-2^{1-\beta+s}\right),\\[4.0pt]
A_1(i,s,1)=C_{\beta,s}\left(i^{1-\beta+s}-(i-1)^{1-\beta+s}-(1-\beta+s)i^{-\beta+s}\right),\\[4.0pt]
A_4(i,s,M+1)=C_{\beta,s}\Big((M+1-i)^{1-\beta+s}-(M-i)^{1-\beta+s}\\
~~~~~~~~~~~~~~~~~~-(1-\beta+s)(M+1-i)^{-\beta+s}\Big),
\end{array}
\end{eqnarray}
where $C_{\beta,s}:=\frac{h^{-\beta} }{(\beta-s)(1-\beta+s)}$.
By (\ref{stiffenessmatrix1}) and (\ref{stiffenessmatrix2}), it holds that $h_{i,j}=h_{j,i}$, i.e.,  matrix  ${\bf H}$   is  symmetric.
As for $B_1(i)$ and $B_2(i)$,  when $\lambda=0$,  we have
\begin{eqnarray}\label{integration0}
B_1(i)=\frac{(x_i-a)^{-\beta}}{\beta},~~~B_2(i)=\frac{(b-x_i)^{-\beta}}{\beta};
\end{eqnarray}
when $\lambda>0$, it holds that
\begin{eqnarray}
&&B_1(i)=\frac{e^{-\lambda(x_i-a)}}{\beta(x_i-a)^\beta}+\frac{\lambda}{\beta(1-\beta)}\frac{e^{-\lambda(x_i-a)}}{(x_i-a)^{\beta-1}}\nonumber\\
&&~~~~~+\lambda^{\beta}\Gamma(-\beta)+\frac{\lambda^2}{\beta(1-\beta)}\int_0^{x_i-a}e^{-\lambda t}t^{1-\beta} dt,\label{integration1}\\
&&\int_0^{x_i-a}e^{-\lambda t}t^{1-\beta} dt=\left(\frac{x_i-a}{2}\right)^{2-\beta}\int_{-1}^1e^{-\frac{\lambda(x_i-a)}{2}(1+\xi)}
(1+\xi)^{1-\beta}d\xi,\label{integration2}
\end{eqnarray}
and
\begin{eqnarray}
&&B_2(i)=\frac{e^{-\lambda(b-x_i)}}{\beta(b-x_i)^\beta}+\frac{\lambda}{\beta(1-\beta)}\frac{e^{-\lambda(b-x_i)}}{(b-x_i)^{\beta-1}}\nonumber\\
&&~~~~~+\lambda^{\beta}\Gamma(-\beta)+\frac{\lambda^2}{\beta(1-\beta)}\int_0^{b-x_i}e^{-\lambda t}t^{1-\beta} dt,\label{integration3}\\
&&\int_0^{b-x_i}e^{-\lambda t}t^{1-\beta} dt=\left(\frac{b-x_i}{2}\right)^{2-\beta}\int_{-1}^1e^{-\frac{\lambda(b-x_i)}{2}(1+\xi)}
(1+\xi)^{1-\beta}d\xi.\label{integration4}
\end{eqnarray}
The integrals  in (\ref{integration2}) and  (\ref{integration4}) can be calculated  by  the Jacobi-Gauss quadrature with
 the weight function  $(1-\xi)^0 (1+\xi)^{1-\beta}$) \cite[Appendix A, p. 447]{Hesthaven:08} and \cite{Shen:11}. Since $e^{-\frac{\lambda(x_i-a)}{2}(1+\xi)}$ and
  $e^{-\frac{\lambda(b-x_i)}{2}(1+\xi)}$ are sufficiently smooth in $[-1,1]$, these calculations yield the spectral accuracy.
  We assume that $B_1(i)$ and $B_2(i)$ are exact in the following  analysis.
    By (\ref{coeffieedddeddee1}), (\ref{coeffieedddeddee2}), (\ref{coeffieedddeddee3}), (\ref{coeffieedddeddee4}), (\ref{stiffenessmatrix1}),
    and (\ref{stiffenessmatrix222}),
it holds that
\begin{lemma} \label{lemma3233}
The entries of matrix ${\bf H}$ satisfies
\begin{eqnarray}\label{ceoffcientproperties}
h_{i,j}<0 \,\,(j\not=i),~~~ h_{i,i}>0, ~~~~ h_{i,i}+\sum_{j=1,j\not= i}^M h_{i,j}>B_1(i)+B_2(i).
\end{eqnarray}
\end{lemma}
According to the Gersgorin theorem \cite[Theorem 4.4]{Axelssion:96}, the minimum eigenvalue of $\bf {H}$ satisfies
\begin{eqnarray}\label{egenvalue1}
\lambda_{min}({\bf {H}})> \min_{1\le i\le M} \left(B_1(i)+B_2(i)\right)>0.
\end{eqnarray}
Thus  ${\bf H}$ is a strictly diagonally dominant  $M$-matrix \cite[Lemma 6.2]{Axelssion:96} and a  symmetric positive definite (s.p.d.) matrix. Therefore, the scheme
 (\ref{matrixform}) has an unique solution. Define the discrete $L_2$ inner product and norms:
\begin{eqnarray}
\left(v,w\right)=h\sum_{i=1}^{M} v_i w_i,~~~~\left\|v\right\|=\sqrt{(v,v)},~~~\left\|v\right\|_{\infty}=\max_{1\le i\le M}\left|v_i\right|.
\end{eqnarray}
\begin{theorem}\label{theoremsection11}
For the scheme (\ref{matrixform}), the following hold.
\begin{enumerate}
\item Let $\beta\in (0,1), s=s_1=0$, and $u(x)\in C^2(\bar{\Omega})$. Then
\begin{eqnarray}\label{LManorm1}
\left\|U-U_h\right\|\le C_1h^{2-\beta},~~~\left\|U-U_h\right\|_{\infty}\le C_2h^{2-\beta},
\end{eqnarray}
where $C_1$ and $C_2$ may depend on the values of $u^{(k)}(x), k=0,1,2$ on $\bar{\Omega}$, but independent of $h$.
\item Let $\beta\in (0,1),\, s=s_1=1$, and $u(x)\in C^3(\bar{\Omega})$. Then
\begin{eqnarray}\label{LManorm2}
\left\|U-U_h\right\|\le C_1h^{2},~~~\left\|U-U_h\right\|_{\infty}\le C_2h^{2},
\end{eqnarray}
where $C_1$ and $C_2$ may depend on the values of $u^{(k)}(x), k=0,1,2,3$ on $\bar{\Omega}$, but independent of $h$.
 \end{enumerate}
\end{theorem}

\begin{proof}
Firstly, taking an inner product of (\ref {matrixform}) with $U_h$, and using the Cauchy-Schwarz inequality, we have
\begin{eqnarray}
\lambda_{\min}({\bf H})\left\|U_h\right\|^2\le \left({\bf H}U_h,U_h\right)\le \left\|F\right\|\left\|U_h\right\|.
\end{eqnarray}

Since
\begin{eqnarray}
&&B_1(i)+B_2(i)=\int_{-\infty}^a\frac{e^{\lambda(y-x_i)}}{(x_i-y)^{1+\beta}}dy+\int_b^{\infty}\frac{e^{\lambda(x_i-y)}}{(y-x_i)^{1+\beta}}dy,~~~~~~~~~~~~~~~\nonumber\\
&&~~~~~~~~~~~~\ge \int_{a-\delta}^a\frac{e^{\lambda(y-x_i)}}{(x_i-y)^{1+\beta}}dy+\int_b^{b+\delta}\frac{e^{\lambda(x_i-y)}}{(y-x_i)^{1+\beta}}dy\nonumber\\
&&~~~~~~~~~~~~\ge e^{-\lambda(b-a+\delta)}\left(\int_{a-\delta}^a\frac{1}{(b-y)^{1+\beta}}dy+\int_b^{b+\delta}\frac{1}{(y-a)^{1+\beta}}dy\right)\nonumber\\
&&~~~~~~~~~~~~=\frac{2e^{-\lambda(b-a+\delta)}(b-a)^{-\beta}}{\beta},\label{sumdiangonal1}
\end{eqnarray}
where $\sigma$  is a given positive real number.
Thus, by (\ref{egenvalue1}), it holds that
\begin{eqnarray}\label{stabilty1}
\left\|U_h\right\|\le \frac{e^{\lambda(b-a+\delta)}(b-a)^{\beta}\beta}{2} \left\|F\right\|.
\end{eqnarray}

Secondly, assume $\left\|U_h\right\|_{\infty}=\left|U_{i_0}\right|$ with $1\le i_0\le M$. By Lemma \ref{lemma3233}, we have
\begin{eqnarray}
&&U_{i_0}\left(\sum_{j=1}^{i_0-1} h_{i_0,j} U_j+\left(h_{i_0,i_0}-\left( B_1(i_0)+B_2(i_0)  \right)\right)U_{i_0}+\sum_{j=i_0+1}^M h_{i_0,j}U_j\right)~~~~~~~~~~~~~\nonumber\\
&&~~~\ge \sum_{j=1}^{i_0-1} h_{i_0,j}U_{i_0}^2+\left(h_{i_0,i_0}-\left( B_1(i_0)+B_2(i_0)  \right)\right)U^2_{i_0}+\sum_{j=i_0+1}^M h_{i_0,j}U_{i_0}^2\nonumber\\
&&~~~\ge 0.
\end{eqnarray}
Then
\begin{eqnarray}
 &&\left(B_1(i_0)+B_2(i_0)\right)\left\|U_h\right\|_{\infty}\nonumber\\
 &&~=\left(B_1(i_0)+B_2(i_0)\right) \left|U_{i_0}\right|\nonumber\\
 &&~\le \Big|  \sum_{j=1}^M h_{i_0,j} U_j\Big|=\left|F_{i_0}\right|.
\end{eqnarray}
Therefore,
\begin{eqnarray}
&&\left\|U_h\right\|_{\infty}\le \frac{1}{B_1(i_0)+B_2(i_0)}\left|F_{i_0}\right|\nonumber\\
&&~~~~~~~~~\le  \frac{e^{\lambda(b-a+\delta)}(b-a)^{\beta}\beta}{2} \left\|F\right\|_{\infty}. \label{stabilty2}
\end{eqnarray}

Finally, by (\ref{exactequation}) and (\ref{matrixform}), it holds that
\begin{eqnarray}
{\bf H}\left(U-U_h\right)=R_h.
\end{eqnarray}
Then the desired results follow from (\ref{stabilty1}), (\ref{stabilty2}), and  (\ref{localerror1}).~~~~~~~~~~~~~~$\Box$
\end{proof}

\begin{remark}
If  $g(x)\not=0$,  then $d_1(i)$ and $d_2(i)$  usually  should also be  calculated by numerical integrations.
 Since $g(x)$ is known, with some of today's well-developed algorithms and software \cite{Abramowitz:65,Gradshteyn:80,Shen:11,software1}, they may be calculated directly or adaptively with the  accuracy not less than the order of the local truncation errors. Thus,  Theorem \ref{theoremsection11} still holds.
\end{remark}

\subsection{Algorithm implementation}\label{section3}
This section focuses on the effective algorithm implementation.
\subsubsection{Structure of the stiffness matrix}
 A symmetric matrix $T_{M}$ is called a symmetric Toeplitz matrix\index{Toeplitz matrix} if its entries are constant along each diagonal, i.e.,
\begin{equation} \label{teoplitzone}
 {\bf T}_{M}=\left [ \begin{matrix}
                      t_0  &  t_{1}   &   \cdots     &  t_{M-2}   & t_{M-1} \\
                       t_1  &  t_0   &    t_{1}     &  \ddots    &   t_{M-2}    \\
                      \vdots   &   t_1 &    \ddots     &  \ddots     & \vdots   \\
                    t_{M-2}   &     \ddots      &      \ddots         &    t_0      &  t_{1}        \\
                       t_{M-1}   &      t_{M-2}     &    \cdots          &  t_1  &  t_0
 \end{matrix}\right ].
 \end{equation}
  The Toeplitz matrix ${\bf T}_{M}$ is circulant if $t_{k}=t_{M-k}$ for all $1\le k\le {M-1}$.
 It should be noticed that a symmetric Toeplitz matrix is determined by its first column (or first row). Therefore, we can store ${\bf T}_M$
 with $M$ entries. Moreover, the product of matrix ${\bf T}_M $ with a vector $V\in \mathbb{R}^M$ can be performed by FFT  in $\mathcal{O}(M\log M)$
 arithmetic operations \cite[pp. 11-12]{Chan:07} and \cite{Jia:15,Lei:13,Lin:14}.

  From (\ref{stiffenessmatrix1}) and (\ref{stiffenessmatrix2}), it is easy to see that except the main diagonal,  the entries of
   matrix ${\bf H}$ are constant along each diagonal. Let ${\bf D}$ denotes the main diagonal matrix of ${\bf H}$. Then
 \begin{eqnarray}
 {\bf H}= {\bf D}+\left({\bf H}-{\bf D}\right),
 \end{eqnarray}
 and  ${\bf H}-{\bf D}$ is a symmetric Teoplitz matrix. Therefore, we can store ${\bf H}$ with $2M$ entries, and calculate ${\bf H} V$ by ${\bf D}V+\left({\bf H}-{\bf D}\right)V$   with  the cost $\mathcal{O}(M\log M)$.

 \subsubsection{ The fast conjugate gradient method}
  The matrix ${\bf H}$ is fully dense due to the nonlocal property of the tempered fractional Laplacian.
The $\mathcal{O}(M^3)$ operations are required to solve the linear system (\ref{matrixform}) by  a direct method. Since the product ${\bf H}V$ can be effectively computed in $\mathcal{O}(M\log M)$, the Krylov subspace iterative methods such as the conjugate gradient (CG) method naturally provide feasible and economical choices for solving such linear systems. These iterative methods
 only require a few matrix-vector products at each step, so they can be conveniently accomplished in $\mathcal{O}(M\log M)$ operations if the total number of
 iteration steps needed for achieving their convergence is not too large.

It is well known that the convergence speed  of the CG method is influenced
 by the condition number, or more precisely,  the eigenvalue
  distribution of ${\bf H}$; the more clustered around the unity the eigenvalues are,
  the faster the convergence rate will be \cite[pp. 8-10]{Chan:07}.
By the Gersgorin theorem,  (\ref{ceoffcientproperties}), and (\ref{sumdiangonal1}), it holds that
\begin{eqnarray}
\lambda_{min}({\bf {H}})> B_1(i)+B_2(i)\ge\frac{2e^{-\lambda(b-a+\delta)}(b-a)^{-\beta}}{\beta}
\end{eqnarray}
and
\begin{eqnarray}
 && \lambda_{\max}({\bf H})\le \max_{1\le i\le M} \Bigg(h_{i,i}-\sum_{j=1,j\not= i}^M h_{i,j}\Bigg)\le \max_{1\le i\le M}( 2h_{i,i}-B_1(i)-B_2(i))\nonumber\\
 &&~~~~~~~\le  \max_{1\le i\le M}\Bigg(2\mathcal{A}_1\bigg(s,\,\frac{e^{\lambda(y-x_i)}}{(x_i-y)^s}\bigg) +2\mathcal{A}_2\bigg(s,\,
 \frac{e^{\lambda(x_i-y)}}{(y-x_i)^s}\bigg)\nonumber\\
 &&~~~~~~~~+\frac{4h^{-\beta}e^{-\lambda h}}{s_1+1-\beta}+B_1(i)+B_2(i) \Bigg).
 \end{eqnarray}
 Noting that
 \begin{eqnarray}
 &&\mathcal{A}_1\left(s,\,\frac{e^{\lambda(y-x_i)}}{(x_i-y)^s}\right)\le \int_{x_1}^{x_{i-1}}(x_i-y)^{-1-\beta}dy+\int_{x_{i-2}}^{x_{i-1}}
 (x_i-x_{i-1})^{-s}(x_i-y)^{s-1-\beta}dy\nonumber\\
 &&~~~~~~~~~~~~~~~~~~~~~~~~~~~~~~\le h^{-\beta}\left(\frac{1-(i-1)^{-\beta}}{\beta}+ 1\right)\le C h^{-\beta},\nonumber\\
 &&\mathcal{A}_2\left(s,\,\frac{e^{\lambda(x_i-y)}}{(y-x_i)^s}\right)\le
 \int_{x_{i+1}}^{x_{i+2}}(x_{i+1}-x_{i})^{-s}(y-x_i)^{s-1-\beta}dy+\int_{x_{i+1}}^{x_M}(y-x_i)^{-1-\beta}dy\nonumber\\
  &&~~~~~~~~~~~~~~~~~~~~~~~~~~~~~~\le h^{-\beta}\left(\frac{1-(M-i)^{-\beta}}{\beta}+ 1\right)\le C h^{-\beta},\nonumber
\end{eqnarray}
and  $B_1(i)\le \frac{h^{-\beta}}{\beta},\, B_2(i)\le  \frac{h^{-\beta}}{\beta}$, we have
\begin{eqnarray}
\lambda_{\max}({\bf H})\le C h^{-\beta}.
\end{eqnarray}
Furthermore, defining
\[V_1:=\Big(\underbrace{0,0,\cdots,0}_{i-1}, 1,0,\underbrace{\cdots,0}_{M-i}\Big), ~~V_2:=\Big(\underbrace{1,1,1,\cdots,1, 1}_M\Big)\]
 and using the Courant-Fischer  theorem  \cite[Theorem 1.5]{Chan:07}
\begin{eqnarray}\label{Courantfishcher}
 \lambda_{\min}({\bf H})=\min_{V\not=0}\frac{\left({\bf H}V, V\right)}{\left(V,V\right)},~~~\lambda_{\max}({\bf H})=\max_{V\not=0}
 \frac{\left({\bf H}V V\right)}{\left(V,V\right)},
\end{eqnarray}
it follows that
\begin{eqnarray}
\lambda_{\max}({\bf H})\ge \frac{\left({\bf H}V_1, V_1\right)}{\left(V_1,V_1\right)}=h_{i,i}\ge \frac{2h^{-\beta}e^{-\lambda h}}{s+1-\beta}\ge Ch^{-\beta}
\end{eqnarray}
and
\begin{eqnarray}
&&\lambda_{\min}({\bf H})\le \frac{\left({\bf H}V_2, V_2\right)}{\left(V_2,V_2\right)}=\frac{1}{M}\sum_{i=1}^M \Big(h_{ii}+\sum_{j\not= i}^M h_{i,j}
 \Big)\nonumber\\
&&\le \frac{1}{M}\left(\sum_{i=2}^M\int_{x_0}^{x_1}\frac{(x_i-y)^{s-1-\beta}}{(ih)^{s}}dy
+\sum_{i=1}^{M-1}\int_{x_M}^{x_{M+1}}\frac{(y-x_i)^{s-1-\beta}}{\left((M+1-i)h\right)^{s}}dy\right)\nonumber\\
&&~+\frac{1}{M}\sum_{i=1}^M\left(\int_{-\infty}^a\frac{1}{(x_i-y)^{1+\beta}}dy
+\int_b^{\infty}\frac{1}{(y-x_i)^{1+\beta}}dy\right)+\frac{2h^{-\beta}}{M(s+1-\beta)}\nonumber\\
&&\le \frac{4h}{b-a}\left(\int_h^{Mh}y^{-1-\beta}dy+\frac{h^{-\beta}}{2^s}+\sum_{i=1}^M \frac{(ih)^{-\beta}}{\beta}+ \frac{h^{-\beta}}{(s+1-\beta)}\right).
\nonumber\\
&&\le \frac{4h}{b-a}\left( \frac{h^{-\beta}-(Mh)^{-\beta}}{\beta}+\frac{h^{-\beta}}{2^s}+\frac{1}{\beta h}\int_0^{Mh}y^{-\beta}dy+\frac{h^{-\beta}}
{(s+1-\beta)} \right)\nonumber\\
&&\le \frac{4}{b-a}\left(\frac{1}{\beta}+\frac{1}{2^s}+\frac{(b-a)^{1-\beta}}{\beta(1-\beta)}+\frac{1}{s+1-\beta}\right), \label{eigulew112}
\end{eqnarray}
where
\begin{eqnarray}
&&\sum_{i=2}^M\int_{x_0}^{x_1}\frac{(x_i-y)^{s-1-\beta}}{(ih)^{s}}dy\nonumber\\
 &&~\le (2h)^{-s}\int_{x_0}^{x_1}(x_2-y)^{s-1-\beta}dy+\sum_{i=2}^{M-1}\int_{x_0}^{x_1}(x_i-y)^{-1-\beta}dy\nonumber\\
 &&~\le \frac{h^{-\beta}}{2^s}+\int_{h}^{Mh}y^{-1-\beta}dy
 \end{eqnarray}
 and
 \begin{eqnarray}
 &&\sum_{i=1}^{M-1}\int_{x_M}^{x_{M+1}}\frac{(y-x_i)^{s-1-\beta}}{\left((M+1-i)h\right)^{s}}dy\nonumber\\
 &&~\le \sum_{i=2}^{M-1}\int_{x_M}^{x_{M+1}}(y-x_i)^{-1-\beta}dy+(2h)^{-s}\int_{x_M}^{x_{M+1}}(y-x_{M-1})^{s-1-\beta}dy\nonumber\\
 &&~\le \frac{h^{-\beta}}{2^s}+\int_{h}^{Mh}y^{-1-\beta}dy
 \end{eqnarray}
 have been used. Thus, $\lambda_{\min}\sim 1$ and $ \lambda_{\max}\sim h^{-\beta}$. As $h$  becomes small, the eigenvalues of ${\bf H}$
  distribute in a very large interval of length  $Ch^{-\beta}$.  Therefore, efficient preconditioning is required to speed up the convergence
   of CG iterations, that is, instead of solving the original
system ${\bf H}U_h=F$, we find a s.p.d. matrix ${\bf B}={\bf L} {\bf L}^{\rm T}$  and solve the preconditioned system
   \begin{eqnarray}\label{preconditionsystem}
   {\bf H}^* U^*=F^*,
   \end{eqnarray}
where ${\bf H}^* ={\bf L}^{-1}{\bf H}{\bf L}^{-{\rm T}},\, U^*={\bf L}^{\rm T} U_h$ and $ F^*={\bf L}^{-1}F$. We require that ${\bf B}$
`near' to ${\bf H}$ in some sense, such that the eigenvalue distributions of ${\bf H}^*$ is clustered  compared to ${\bf {H}}$.
In the following, we consider two types of preconditioners:

 Firstly, since for $j\le i-2$, we have
 \begin{eqnarray}
 &&\left|h_{i,j}\right|=\left(A_1(i,j+1)+A_2(i,j)\right){e^{-\lambda (i-j)h}}\,{(i-j)^{-s}}\nonumber\\
 &&~~~~~~~~=\frac{e^{-\lambda (i-j)h}}{h^{s+1}(i-j)^{s}}\Bigg(\int_{x_j}^{x_{j+1}}(x_{j+1}-y)(x_i-y)^{s-1-\beta}dy\nonumber\\
 &&~~~~~~~~~~~+\int_{x_{j-1}}^{x_j} (y-x_{j-1})(x_i-y)^{s-1-\beta}dy\Bigg)\nonumber\\
&&~~~~~~~~\le C(i-j)^{-s}(i-j-1)^{s-1-\beta} h^{-\beta}e^{-\lambda (i-j)h};
 \end{eqnarray}
if $|i-j|$ is sufficiently large, the entries $h_{i,j}$ are very small relative to the one near the main diagonal (with the order $h^{-\beta}$).
Hence, similar to \cite{Lin:14,Zhao:16}, we define a  symmetric  $(2k+1)$-bandwidth   matrix 
 \begin{equation} \label{teoplitzoneLU}
 {\bf G}:=\left [ \begin{matrix}
                      h_{1,1} &  \cdots    &   h_{1,k+1}      &            &  \\
                      \vdots  &    \ddots  &                &  \ddots    &       \\
                    h_{i,i-k} &            &    h_{i,i}     &            & h_{i,i+k}   \\
                              &   \ddots   &               &   \ddots    &  \vdots        \\
                              &  h_{M,M-k} &    \cdots     &             & h_{M,M}
 \end{matrix}\right ]+{\bf O}
 \end{equation}
 and expect ${\bf G}$  to be a reasonable approximation  of ${\bf H}$. Here ${\bf O}$ is a diagonal matrix satisfying
 ${\bf H} e= {\bf G} e,\, e=(1,1,\cdots,1)$, which is a  common  technique (i.e., the so-called  diagonal compensation)
 in designing preconditioners for $M$-matrices \cite[Sections 6 and 7]{Axelssion:96}.  By Lemma \ref{lemma3233} and (\ref{egenvalue1}),
  ${\bf G} $   is a s.p.d. $M$-matrix. Thus,  we can  perform its incomplete Cholesky (ichol)  factorization to generate a  banded matrix ${\bf L}$.
   We desire  that matrix  ${\bf B}={\bf L}{\bf L}^{\rm T} $  serves as an effective preconditioner of ${\bf H}$.

Secondly,  T. Chan's (optimal) circulant preconditioner has been widely used in solving the  Toeplitz systems \cite{Chan:07,Lei:13}. For the Toeplitz matrix ${\bf{T}}_M$ defined in (\ref{teoplitzone}),  the  entries in the first column of the T. Chan circulant preconditioner $C_F({\bf T}_M)$ are given by
\begin{eqnarray}
c_k=\frac{(M-k)t_k+k\,t_{M-k}}{M},& ~~0\le k\le M-1.
\end{eqnarray}  However, matrix ${\bf H}$ here
 may not be  a  Toeplitz matrix
(due to the entries on the  main diagonal), we can not construct the  T. Chan circulant preconditioner directly.
Recalling that   the generation process of $h_{i,i}$,   it holds that  $h_{i,i}=h_{M+1-i,M+1-i}$ for $i=1,2,\cdots, \lfloor\frac{M}{2}\rfloor$, and
\begin{eqnarray}
&&h_{i+1,i+1}-h_{i,i}=B_1(i+1)+B_2(i+1)-B_1(i)-B_2(i)\nonumber\\
&&~~~~~~~~+\int_{a}^{x_1}I_{\left[a,x_1\right]}\left({e^{-\lambda(x_{i+1}-y)}}{\left(x_{i+1}-y\right)^{-s}}\right)(x_{i+1}-y)^{s-1-\beta}dy\nonumber\\
&&~~~~~~~~-\int_{x_{M}}^bI_{\left[x_M,b\right]}\left({e^{-\lambda(y-x_{i})}}{\left(y-x_i\right)^{-s}}\right)(y-x_{i})^{s-1-\beta}dy\nonumber\\
&&~~~~~~~=\int_{x_i-a}^{x_{i+1}-a}\left(I_{\left[x_i-a,x_{i+1}-a\right]}\left({e^{-\lambda t}}{t^{-s}}\right)-{e^{-\lambda t}}{t^{-s}}\right) t^{s-1-\beta}dy\nonumber\\
&&~~~~~~~~-\int_{b-x_{i+1}}^{b-x_i}\left(I_{\left[b-x_{i+1},b-x_i\right]}\left({e^{-\lambda t}}{t^{-s}}\right)-{e^{-\lambda t}}{t^{-s}}\right)t^{s-1-\beta}dy
\end{eqnarray}
for $i=1,2,\cdots,\lceil \frac{M}{2}\rceil-1$, where $\lfloor\beta\rfloor:=\left\{ z\in \mathbb{N}: 0<z-\beta\le 1\right\}$,
 and the definition of $I_{[a_1,b_1]}$ is  given in (\ref{edde33sseee}).
We have the following observations:
when  $\lambda=s=0$, then $h_{i,i}=h_{i+1,i+1}$,
 which means that matrix ${\bf H}$  actually is a Toeplitz matrix;
 when  $s=0$ and $\lambda>0$,  by $\left|\left(e^{-\lambda t}\right)^{(2)}\right|=\left|\lambda^2 e^{-\lambda t}\right|\le \lambda^2$ ($t\ge 0$),
it follows that
\begin{eqnarray}
\left|h_{i+1,i+1}-h_{i,i}\right|\le  C h^{2-\beta}<< h_{i,i}\sim h^{-\beta}~~~{\rm for~} h\to 0.
\end{eqnarray}
 Though  it is not easy to prove that the changes of the entries of the main diagonal  of $ {\bf H}$ are slow  for  the cases $s=1$,
 the numerical results show they actually do.  These inspire us to construct a Toeplitz matrix ${\bf G}$ as
 \begin{eqnarray}
{\bf G}=\frac{\sum_ {i=1}^M {h_{i,i}}}{M}{\bf I}+\left({\bf H}-{\bf D}\right)~~~({\bf I} {~\rm is ~ identity~ matrix})
 \end{eqnarray}
 and expect  the corresponding  matrix  $C_F({\bf G})$ to be an effective  preconditioner of ${\bf H}$.

The algorithms of the CG  method and the preconditioned CG (PCG) method can be founded in \cite[pp. 470-473]{Axelssion:96}.
At each iteration,  the  required product of  ${\bf H} $  with a vector  $V\in \mathbb{R}^M$ can be performed  with the cost $\mathcal{O}(M\log M)$.
 Note that in the PCG algorithm,  the matrices $ {\bf L}$ and  ${\bf L}^{\rm T}$ do not appear explicitly, to perform the preconditioning,
we only need to calculate ${\bf B}^{-1} V$ or to solve the corresponding  equation  ${\bf B} X=V$   effectively.
For the ichol  factorization preconditioner,  the bandwidth characteristic of matrix ${\bf L}$  allows us to solve ${\bf B} X=V$    with the cost
 $\mathcal{O}(kM)$; for the T. Chan circulant preconditioner ${\bf B}=C_F({\bf G})$, ${\bf B}^{-1}V$ can be calculated by the FFT with the cost $\mathcal{O}(M\log M)$
 \cite[pp. 11-12]{Chan:07}. Thus the total cost for each iteration still is $\mathcal{O}(M\log M)$.

\section{Numerical scheme for the case $\beta\in [1,2)$} \label{sec3}
By choosing $s=0, s_1=1$ or $s=s_1=1$ in (\ref{bianliangdaihua1})-(\ref{bianliangdaihua2}), the numerical schemes introduced
 in  Subsection (\ref{subsec21}) can be easily  extended to the cases  $\beta\in [1,2)$.  If $\beta\in (1,2)$,
  the estimates (\ref{localtruncation1})-(\ref{localtruncation3}) still hold and we have
\begin{eqnarray}
 \left|r_h^i\right|=\left\{\begin{array}{l}
 \mathcal{O}(h^{2-\beta}) {~~~~\rm for~}  u\in C^3(\bar{\Omega}), \,s=0, s_1=1, \\
  \mathcal{O}(h^{3-\beta}){~~~~\rm  for~} u\in C^3(\bar{\Omega}), \,s=s_1=1.
  \end{array}\right.
 \end{eqnarray}
If $\beta=1$,  (\ref{localtruncation3}) still holds for $s_1=1$ and (\ref{localtruncation1}) and (\ref{localtruncation2}) are also true for $s=0$, however, for $s=1$, we have the following estimates
 \begin{eqnarray}
&&\left|\int_a^{x_{i-1}}g_1(i,s,y)(x_i-y)^{s-1-\beta}dy-\mathcal{A}_1(s,g_1)\right|\nonumber\\
&&~\le \frac{\big\|g^{(2)}_1(i,s,y)\big\|_{L^\infty[a,x_{i-1}]} h^2}{2}\left|\ln(h)-\ln(x_i-a)\right|
\end{eqnarray}
and
\begin{eqnarray}
&&\left|\int_{x_{i+1}}^b g_2(i,s,y)(y-x_i)^{s-1-\beta} dy-\mathcal{A}_2(s,g_2)\right|\nonumber\\
&&~~\le \frac{\big\|g^{(2)}_2(i,s,y)\big\|_{L^\infty[x_{i+1},b]} h^2}{2}\left|\ln(h)-\ln(b-x_i)\right|;
\end{eqnarray}
then we have that if $\beta=1$,
\begin{eqnarray}
 \left|r_h^i\right|=\left\{\begin{array}{lll}
 \mathcal{O}(h^{2-\beta}) & { for } &  u\in C^3(\bar{\Omega}), \,s=0,s_1=1, \\
  \mathcal{O}(h^2\ln(h)) & { for } & u\in C^3(\bar{\Omega}), \,s=s_1=1.
  \end{array}\right.
 \end{eqnarray}
 With the same  proof process as in  Theorem \ref{theoremsection11}, it holds that
\begin{theorem}\label{theoremsection21}
\begin{enumerate} For the scheme (\ref{matrixform}), we have the following estimates.
\item Let $\beta\in [1,2), s=0,s_1=1$, and $u(x)\in C^3(\bar{\Omega})$. Then
\begin{eqnarray}\label{LManormbeta21}
\left\|U-U_h\right\|\le C_1h^{2-\beta},~~~\left\|U-U_h\right\|_{\infty}\le C_2h^{2-\beta},
\end{eqnarray}
where $C_1$ and $C_2$ may depend on the values of $u^{(k)}(x), k=0,1,2,3$ on $\bar{\Omega}$, but independent of $h$.
\item Let $\beta\in [1,2), s=s_1=1$, and $u(x)\in C^3(\bar{\Omega})$. Then
\begin{eqnarray}\label{LManormbeta21}
\left\|U-U_h\right\|\le C_1h^{3-\beta},~~~\left\|U-U_h\right\|_{\infty}\le C_2h^{3-\beta}
\end{eqnarray}
for $\beta\in (1,2)$, and
\begin{eqnarray}\label{LManormbeta11}
\left\|U-U_h\right\|\le C_1|\ln(h)|h^2,~~~\left\|U-U_h\right\|_{\infty}\le C_2|\ln(h)|h^2
\end{eqnarray}
for $\beta=1$, where $C_1$ and $C_2$ may depend on the values of $u^{(k)}(x), k=0,1,2,3$ on  $\bar{\Omega}$, but independent of $h$.
 \end{enumerate}
\end{theorem}

As for generating the stiffness matrix ${\bf H}$, the  results in (\ref{stiffenessmatrix2}) and (\ref{integration0})-(\ref{integration4})
still hold for $\beta\in (1,2)$. While for
$\beta=1$ and $s=0$, one has
\begin{eqnarray}\label{stiffenessmatrix5}
\begin{array}{l}
A_1(i,0,j+1)+A_2(i,0,j)=A_3(i,0,j+1)+A_4(i,0,j)\\[3pt]
~=\frac{1}{h}\left(2\ln(|i-j|-\ln\left(|i-j|+1\right)-\ln\left(|i-j|-1\right)\right),\\[5pt]
A_2(i,0,i-1)=A_3(i,i+2)=\frac{1}{h}\left(1-\ln(2)\right),\\[5pt]
A_1(i,0,1)=\frac{1}{h}\left(\ln\left(\frac{i}{i-1}\right)-\frac{1}{i}\right),\\[5pt]
A_4(i,0,M+1)=\frac{1}{h}\left(\ln\left(\frac{M+1-i}{M-i}\right)-\frac{1}{M+1-i}\right);
\end{array}
\end{eqnarray}
 and for $\beta=1$ and $s=1$, one has
\begin{eqnarray}\label{stiffenessmatrix6}
\begin{array}{l}
A_1(i,1,j+1)+A_2(i,1,j)=A_3(i,1,j+1)+A_4(i,1,j)\\[3pt]
~=\frac{1}{h}\Big(-2|i-j|\ln(|i-j|)+\left(|i-j|+1\right)\ln\left(|i-j|+1\right)\\
~~~~+\left(|i-j|-1\right)\ln\left(|i-j|-1\right)\Big),\\[5pt]
A_2(i,1,i-1)=A_3(i,i+2)=\frac{1}{h}\left(2\ln(2)-1\right),\\[5pt]
A_1(i,1,1)=\frac{1}{h}\left((1-i)\ln\left(\frac{i}{i-1}\right)+1\right),\\[5pt]
A_4(i,1,M+1)=\frac{1}{h}\left((i-M)\ln\left(\frac{M+1-i}{M-i}\right)+1\right).
\end{array}
\end{eqnarray}
The calculations for $B_1(i)$ with $\beta=1$  are given below:
 when $\lambda=0$, the results in (\ref{integration0}) still hold;
 when $\lambda>0$ and $b-x_i\ge \frac{1}{2\lambda}$, we first rewrite $B_1(i)$ as
  \begin{eqnarray}
   \int_{b-x_i}^{\infty} e^{-\lambda t} t^{-2}dt \approx\int_{\lambda/K}^{\frac{1}{b-x_i}}e^{-\frac{\lambda}{t}}dt
   =\left(\frac{1}{2(b-x_i)}-\frac{\lambda}{2K}\right)\int_{-1}^1
   e^{-\lambda/\eta(\xi, x_i)}d\xi \label{integralbeta11}
  \end{eqnarray}
with $\eta(\xi, x_i)=:\frac{\xi+1}{2(b-x_i)}-\frac{\lambda(\xi-1)}{2K}$ and then calculate $\int_{-1}^1
e^{-\lambda/\eta(\xi, x_i)}d\xi$ by the Jacobi-Gauss quadrature with the weight function $(1-\xi)^0(1+\xi)^0$ \cite{Hesthaven:08,Shen:11}
(In our calculation, $K$ is chosen as $80$);
 when $\lambda>0$ and $b-x_i<\frac{1}{2\lambda}$, we  first rewrite $B_1(i)$ as
   \begin{eqnarray}
   B_1(i)=\frac{e^{-\lambda(b-x_i)}}{b-x_i}-\lambda \int_{\lambda(b-x_i)}^{\infty} \frac{e^{-t}}{t} dt,
   \end{eqnarray}
and then use the  series expansion representation in \cite[Eq. 5.1.11]{Abramowitz:65}, i.e.,
    \begin{eqnarray}
    \int_{\lambda(b-x_i)}^{\infty} \frac{e^{-t}}{t} dt=-\gamma-\ln(z)-\sum_{n=1}^{\infty}\frac{(-1)^n z^n}{n\Gamma(n+1)},~~~z=\lambda(b-x_i),
    \label{integralbeta12}
    \end{eqnarray}
 where $\gamma$ is the Euler constant (in our calculations,  the series is truncated with the first $26$ items).

Since matrix $\bf{H}$ has the same structure as the case $\beta\in (0,1)$, the implementation techniques developed in  Section \ref{section3}
can also be used here to solve the corresponding algebraic equation, and the numerical results show  they still work well.

\section{Numerical results}\label{sec4}

In this section, we make some numerical experiments to show the performance of numerical schemes above. All are run in
MATLAB 7.11  on a PC with Intel(R) Core (TM)i7-4510U 2.6 GHz processor and 8.0 GB RAM. For the CG and PCG iterations,  we adopt the initial guess $U_0=0$
 and the stopping criterion
 \[ \frac{\|r(k)\|_{l_2}}{\|r(0)\|_{l_2}}\le 1e-9,\]
 where $r(k)$ denotes the residual vector  after $k$ iterations. Let $h_1=(b-a)/(M_1+1)$ and $h_2=(b-a)/(M_2+1)$. The convergence rates at $M=M_1$ are computed by
 \begin{eqnarray}
 {\rm rate}=\left\{\begin{array}{ll}
\frac{
\ln\left[\left(\ln(h_2)\left({\rm the ~error~ at~h_1}\right)\right)/\left(\ln(h_1)\left({\rm the~ error~ at~ h_2}\right)\right)\right]}
{\ln\left(h_1/h_2\right)},&s=s_1=1, \beta=1,\\[5pt]
 \frac{
~~~\ln\left({\rm the ~error~ at~h_1}/{\rm the~ error~ at~ h_2}\right)}{\ln\left(h_1/h_2\right)},& {\rm~ otherwise.}
 \end{array}\right.
 \end{eqnarray}

\begin{example}\label{example1}
Consider  model (\ref{Laplacedirichelt}) with $g(x)=0$, and the force term  $f(x)$ being derived from the exact solution
$u(x)=x^2(1-x)$ for $x\in \Omega$.
\end{example}
Note that  $u\in C^3(\bar{\Omega})$. If $\lambda=0$,  the explicit form of $f(x)$  is given  in \cite[Example 1]{Zhang:17}.
If $\lambda\not=0$,  the value of $f(x)$  at $x_i$ should  be calculated numerically.
More specifically, for $\beta\not=1$, we have
\begin{eqnarray}
&&f(x_i)=u(x_i)\left(B_1(i)+B_2(i)\right)\nonumber\\
&&~~~~~~~+\frac{2x_i-3x_i^2}{1-\beta}\left(x_i^{1-\beta}e^{-\lambda x_i}-(1-x_i)^{1-\beta}e^{-\lambda(1-x_i)}\right)\nonumber\\
&&~~~~~~~+\int_0^{x_i}\left(-t+3x_i-1+\frac{\lambda(2x_i-3x_i^2)}{1-\beta}\right)e^{-\lambda t}t^{1-\beta}dt\nonumber\\
&&~~~~~~~+\int_0^{1-x_i}\left(t+3x_i-1-\frac{\lambda(2x_i-3x_i^2)}{1-\beta}\right)e^{-\lambda t} t^{1-\beta}dt,\label{example1integral1}
\end{eqnarray}
and the integrals in (\ref{example1integral1}) can be handled as in (\ref{integration2}) and  (\ref{integration4}); for $\beta=1$, we have
\begin{eqnarray}
&&f(x_i)=u(x_i)\left(B_1(i)+B_2(i)\right)\nonumber\\
&&~~~~~~~~+\int_0^{x_i}(-t+3x_i-1)e^{-\lambda t}dt+\int_0^{1-x_i}(t+3x_i-1)e^{-\lambda t} dt\nonumber\\
&&~~~~~~~~+\left(2x_i-3x_i^2\right)\left(\int_{1-x_i}^{\infty}e^{-\lambda  t}t^{-1}dt-\int_{x_i}^{\infty}e^{-\lambda t} t^{-1} dt\right),\label{example1integral2}
\end{eqnarray}
and the integrals in the second line of  (\ref{example1integral2}) can be handled as in (\ref{integralbeta11}) and (\ref{integralbeta12}).

The  errors and the corresponding convergence rates with different
$s,s_1,\beta, \lambda$,  are listed in Tables \ref{table:1-1}, which confirm the theoretical analysis
in Theorems \ref{theoremsection11} and \ref{theoremsection21}.
The CUP time and the iterative times of the CG and PCG method  are presented in Tables \ref{table:1-2} and \ref{table:1-3},
where ``PCG(Ichol)'' denotes the perconditioners ${\bf B}$  coming from the ichol factorization
of the $(2k+1)$-bandwidth matrix ${\bf G}$ with $k=10$,  and  ``PCG(T)'' denotes that the preconditioner is the T. Chan circulant matrix.
The results show that the CPU time spent with PCG methods  are much less than  those with  the
   Gauss elimination method (the data under ``Gauss'' in Tables \ref{table:1-2} and \ref{table:1-3} ) and the CG method, and  the  T. Chan circulant preconditioner with the iterative times almost independent of $h$ is a little more effective than  the ichol  preconditioner.
   In Figure \ref{figure:1-1}, we display the eigenvalue distribution of the matrix systems with and without preconditioning; after preconditioning, the eigenvalues become clustered around the unity.
\begin{table}[!h t b p]\fontsize{7.0pt}{10.5pt}\selectfont
\begin{center}
 \caption{Errors and convergence rates in  Example \ref{example1} with $M=2^J-1$.}
\begin{tabular}{cc|cc|cc|cc|cc}
  \hline
  $(\beta,\,s,\,s_1)$& $J$   &\multicolumn{4}{c|}{$\lambda=0.5$ }    &\multicolumn{4}{|c}{$\lambda=3$  }\\
     &           & $L^2$-Err      & Rate      &$L^{\infty}$-Err  & rate     &$L^2$-Err  & Rate     &$L^\infty$-Err  & Rate    \\
   \hline

                  &$12$     &2.4068e-06   &  --      &  3.7160e-06    & --    &7.0941e-06     &--       & 1.0474e-05     &--   \\
    $(0.5,\,0,\,0)$  &$13$    &8.5343e-07   & 1.50     & 1.3177e-06    & 1.50       &2.5222e-06  & 1.49     & 3.7238e-06    &1.49  \\
                   &$14$     &3.0234e-07  &  1.50   &  4.6680e-07   &  1.50  &8.9518e-07    &1.49      &1.3216e-06    &1.49  \\

 \\[-5pt]
        &$12$     &  2.6157e-09                &  --         & 4.1651e-09      & --    &2.7141e-08     &--       & 3.9151e-08    &--  \\
    $(0.5,\,1,\,1)$   &$13$     &6.5490e-10   & 2.00       & 1.0428e-09        & 2.00 &6.8022e-09     &  2.00     & 9.8127e-09      &2.00\\
                   &$14$     & 1.6391e-10   &  2.00      &2.6096e-10     &  2.00   &1.7036e-09    & 2.00     & 2.4577e-09  &2.00\\
 \hline

                  &$12$     &1.3524e-05    &  --         &  1.9526e-05   & --    &6.3501e-05    &--       & 8.6565e-05      &-- \\
    $(1.0,\,1,\,0)$   &$13$     &6.7700e-06   & 1.00    & 9.7756e-06      & 1.00     &3.1824e-05      &  1.00     &  4.3391e-05    &1.00\\
                   &$14$     &3.3872e-06 &  1.00    & 4.8911e-06     &  1.00    &1.5932e-05    & 1.00     &2.1725e-05 &1.00 \\

 \\[-5pt]
                     &$12$     &1.0989e-08     &--       &1.6835e-08      & --       &1.2222e-07   &--       & 1.8135e-07     &-- \\
    $(1.0,\,1,1)$   &$13$     & 2.9312e-09    &  2.02   &4.4886e-09      & 2.02       &3.2845e-08     &2.01   &4.8751e-08    &2.01\\
                   &$14$     &7.7258e-10   &  2.03   &1.1850e-09    &  2.02     &8.7695e-09  & 2.01     &1.3024e-08  &2.01
\\
\hline

                  &$12$     &1.2882e-04    &  --        &1.9477e-04       & --     &1.8667e-04    &--       &2.8056e-04     &-- \\
    $(1.5,\,1,\,0)$   &$13$     &9.1256e-05   & 0.50     &1.3797e-04       &0.50    &1.3314e-04     &0.50     & 2.0010e-04     &0.49\\
                   &$14$     &6.4593e-05   & 0.50      &9.7660e-05      &0.50    &9.4560e-05    & 0.50      &1.4212e-04  &0.50 \\

 \\[-5pt]
                    &$12$     & 2.4683e-07   &  --     &  3.7330e-07  & --       &2.1634e-06      &--        & 3.2512e-06     &--   \\
    $(1.5,\,1,\,1)$   &$13$     &8.7460e-08    &1.50     &  1.3226e-07    & 1.50    &7.6763e-07     &1.49      & 1.1536e-06     &1.49\\
                   &$14$     &3.0694e-08    & 1.50      &4.6482e-08   &  1.50     &2.7217e-07     & 1.50     &4.0906e-07     &1.50\\
                       \hline
\end{tabular} \label{table:1-1}
\end{center}
\end{table}

 \begin{table}[!h t b p]\fontsize{7.0pt}{10pt}\selectfont
\begin{center}
 \caption{Performance of the CG and PCG methods in  Example \ref{example1} with $\lambda=0.5$ and $ M=2^J-1$.}
\begin{tabular}{cc|cc|cc|cc|c}
  \hline

$(\beta,\,s,\,s_1)$ & $J$   & \multicolumn{2}{c}{CG }    & \multicolumn{2}{c}{PCG\,(Ichol) }&\multicolumn{2}{c}{PCG\,(T)}  &Gauss  \\
      &        &      $\#$ iter    & CPU(s)    &$\#$ iter       &  CPU(s)      & $\#$ iter     & CPU(s) & CPU(s) \\

                     \hline
              &$12$           & 97    &0.8174           & 40     &0.1472          &  11       &  0.0182   & 1.1997 \\
  $(0.5,\,0,\,0)$   &$13$         & 115  &   0.3248         &44     & 0.1789        &  11       & 0.0627   &  6.0118   \\
              &$14$           & 138    &2.7400          &49      &2.1894          &  11       & 0.1691  & 56.6159\\
          \\[-5pt]

              &$12$            & 74        &  0.1622          & 39    &0.0851       0    & 10         & 0.0159 &0.7742 \\
  $(0.5,\,1,\,1)$  &$13$            &88      &  0.3222       &   43    &0.1619         &  11         &  0.0609&  7.5780   \\
              &$14$           &105       &   3.1492      &  49     & 1.1515          & 11         &  0.1355  &  57.7863 \\

    \hline

               &$12$       &  329    &1.3938           & 47       &0.1632            &  15        &  0.0215  &  0.8483 \\
  $(1.0,\,1,\,0)$   &$13$      &  468   & 1.3587            & 58      &0.2424            &  16        & 0.0888 &6.0763  \\
              &$14$         & 664    &7.3204           &71       &3.3773             &  17       &  0.2877 & 53.1450 \\

     \\[-5pt]
              &$12$       &  337      &   0.6903        &  47   &0.7371         &  15         &0.0715&   0.8431\\
  $(1.0,\,1,1)$    &$13$     &   479       &  1.6629    &  58     &0.2146        & 16       &0.0872& 6.1244\\
              &$14$       & 680      &   7.5843       &  71     & 3.368       & 17       & 0.2486&  53.5024   \\
               \hline

              &$12$    &  1363       &   1.8888       & 30     &0.2089      & 29         &  0.0356  &   0.8860\\
  $(1.5,\,1,\,0)$    &$13$   & 2300        & 8.8453       & 35      & 0.1220          &34         &    0.1779&6.7129   \\
              &$14$    & 3880      & 25.9750       &  42     & 0.4814        &  41        &   0.4571& 55.0814  \\

      \\[-5pt]
              &$12$   & 1383      & 1.5639       &  30      &0.0609        &  29       &   0.0448& 0.9285  \\
  $(1.5,\,1,\,1)$  &$13$  & 2333       &  7.2669      & 35      &0.1538         & 33         &0.1960-& 7.5258   \\
              &$14$   &  3935       &  30.7326       & 42    & 1.0987      & 39       &  0.4440  & 57.1891 \\
            \hline
\end{tabular}\label{table:1-2}
\end{center}
\end{table}

 \begin{table}[!h t b p]\fontsize{7.0pt}{10pt}\selectfont
\begin{center}
 \caption{Performance of the CG and PCG methods in  Example \ref{example1} with $\lambda=3$ and $M=2^J-1$.}
\begin{tabular}{cc|cc|cc|cc|c}
  \hline

$(\beta,\,s,\,s_1)$ & $J$   & \multicolumn{2}{c}{CG } & \multicolumn{2}{c}{PCG\,(Ichol) }  &\multicolumn{2}{c}{PCG\,(T)}  &Gauss~~ \\
  &        &      $\#$ iter    & CPU(s)    &$\#$ iter       &  CPU(s)     & $\#$ iter     & CPU(s) & CPU(s) \\
                     \hline
              &$12$          &  127   & 0.2231             & 70      &0.1079        &  12        &  0.0207 &0.8877 \\
   $(0.5,\,0,\,0)$   &$13$         & 152 & 0.5855             & 88      &0.3419         &  12       &  0.0544&6.8314  \\
              &$14$           & 182   & 5.2566              &108     &0.9474          &  12       & 0.1048 &49.5857  \\
          \\[-5pt]

              &$12$            &97        &  0.5242          & 70    &0.0897       &  11       & 0.0199  & 0.9509 \\
  $(0.5,\,1,\,1)$  &$13$            &116     & 0.3499        &   88      &0.2889        &  12         & 0.0525 &6.8463   \\
              &$14$           &139        &  1.7316         &   107      &0.8856         & 12        & 0.0989 &50.5869   \\

    \hline

               &$12$        &   376      &  0.3920       & 57     &0.0727         & 17         &   0.0269 & 0.8913 \\
  $(1.0,\,1,\,0)$   &$13$       &   534        &1.7202      &  74      & 0.3233          &  17        &   0.1141& 7.2649  \\
              &$14$          &   758      & 7.2397       &   95     & 3.3281         &  19         &    0.2291 & 57.5678   \\

     \\[-5pt]
              &$12$     & 385    &   1.1568        &  57      & 0.1364       &  17       & 0.0610 &  0.8701\\
  $(1.0,\,1,1)$    &$13$    & 547       & 1.8208     & 75     &0.3743      & 17       &  0.0991 &7.1813 \\
              &$14$     & 776     &  12.3254        &  95      &5.1381        & 19        & 0.2976 &  53.7438  \\
               \hline

              &$12$    &  1423        & 1.9613           &  29     & 1.8896        &  30       &   0.0516 & 0.8078 \\
  $(1.5,\,1,\,0)$    &$13$   &  2400        & 6.6799      &  35      &0.2333         & 37        &  0.2130  &6.2326 \\
              &$14$    &   4047        & 25.8704       &  42     &0.4828      &  42        &  0.5635&55.6647\\
      \\[-5pt]
              &$12$   & 1444           & 1.3327       &   29    &1.1708         & 30         &   0.0382 & 0.8009  \\
  $(1.5,\,1,\,1)$   &$13$  &  2435       &  9.4863        &  36     &0.1528          &  37         &   0.1880 & 6.8698  \\
              &$14$   &   4104        & 26.0405      &  42    &0.4805      &  42         &  0.5481&51.1018 \\
            \hline
\end{tabular}\label{table:1-3}
\end{center}
\end{table}

 \begin{figure}[!h t p]
\begin{center}
\includegraphics[width=1.45in,height=1.2in,angle=0]{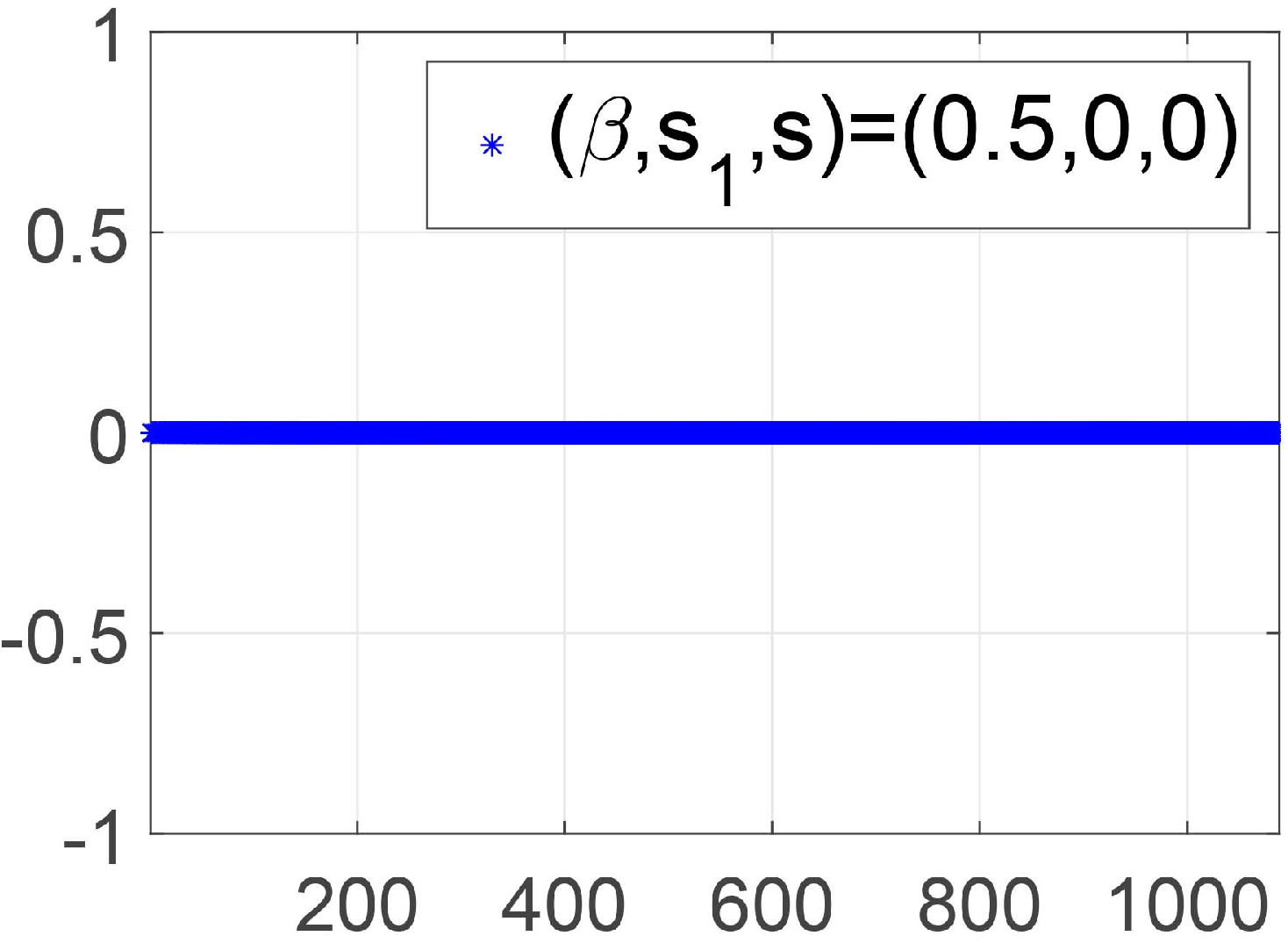}
\includegraphics[width=1.45in,height=1.2in,angle=0]{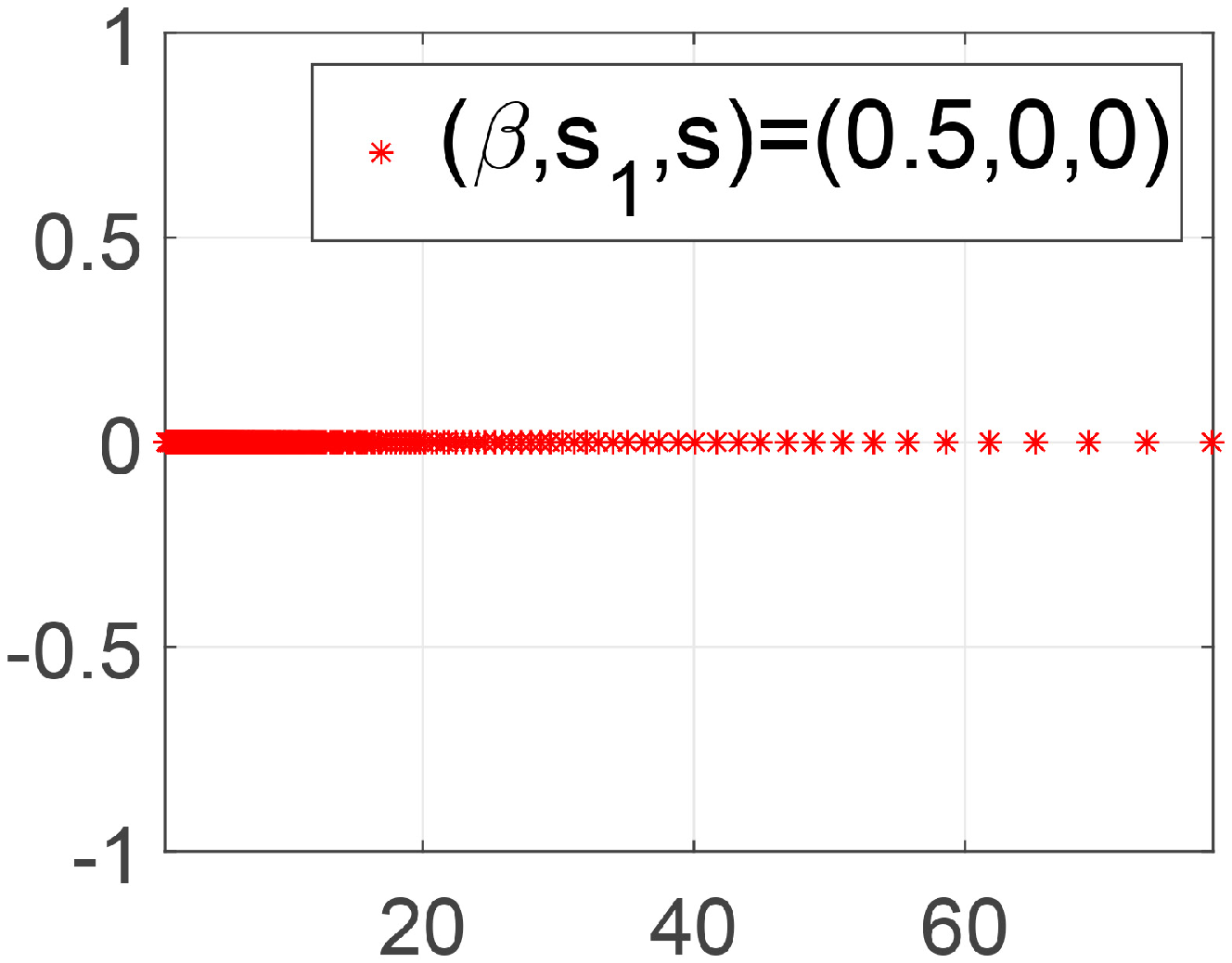}
\includegraphics[width=1.45in,height=1.2in,angle=0]{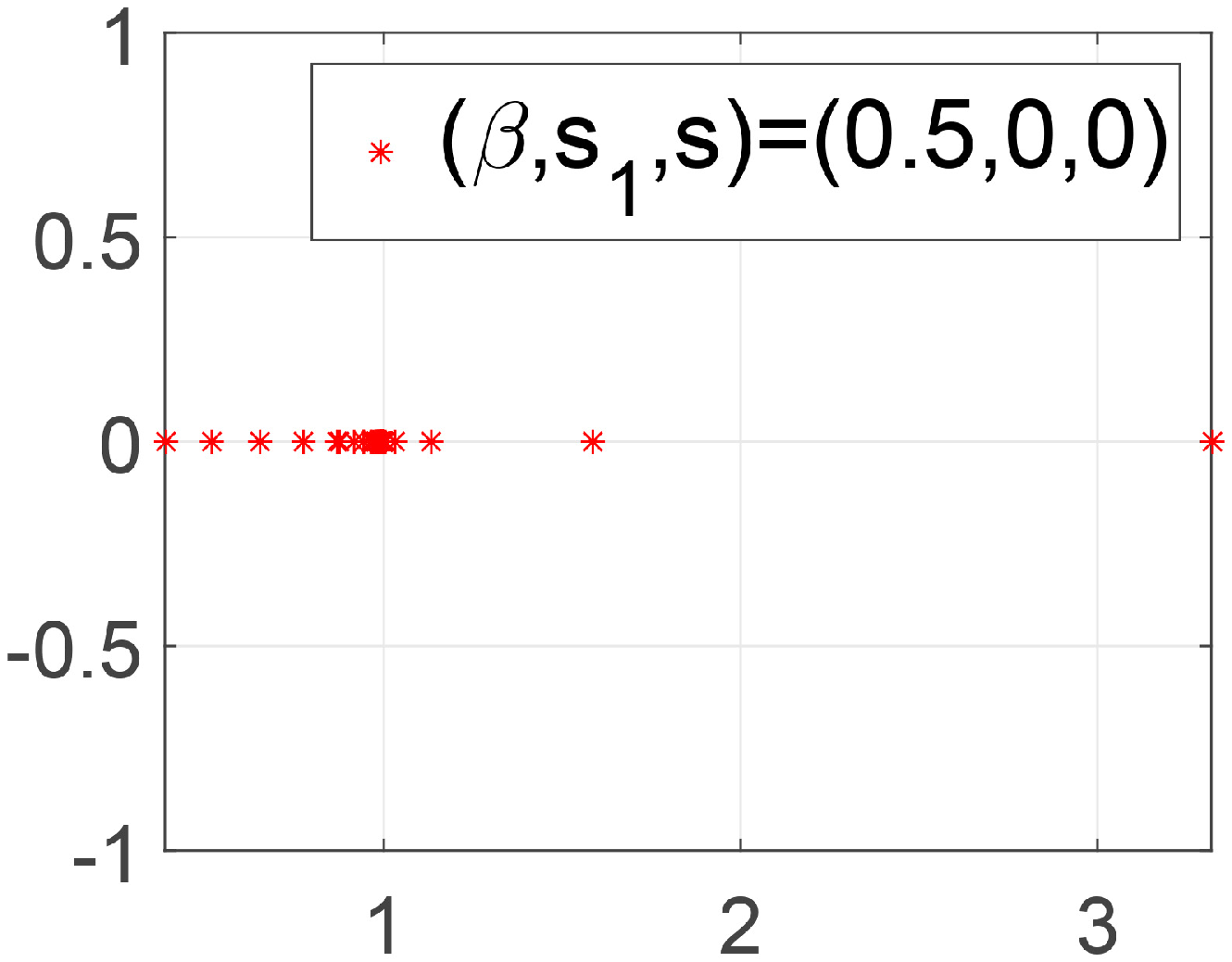}\\
\includegraphics[width=1.45in,height=1.2in,angle=0]{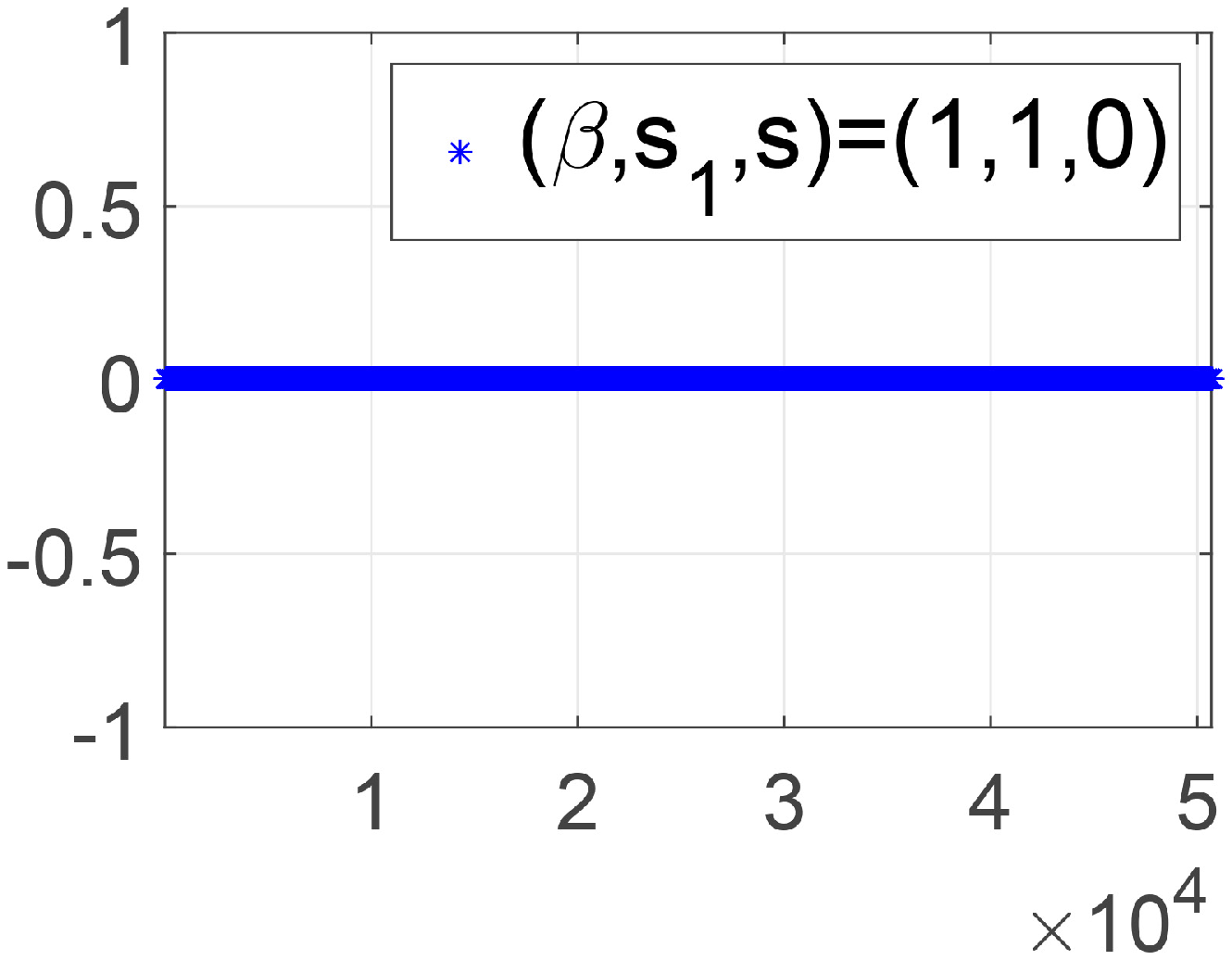}
\includegraphics[width=1.45in,height=1.2in,angle=0]{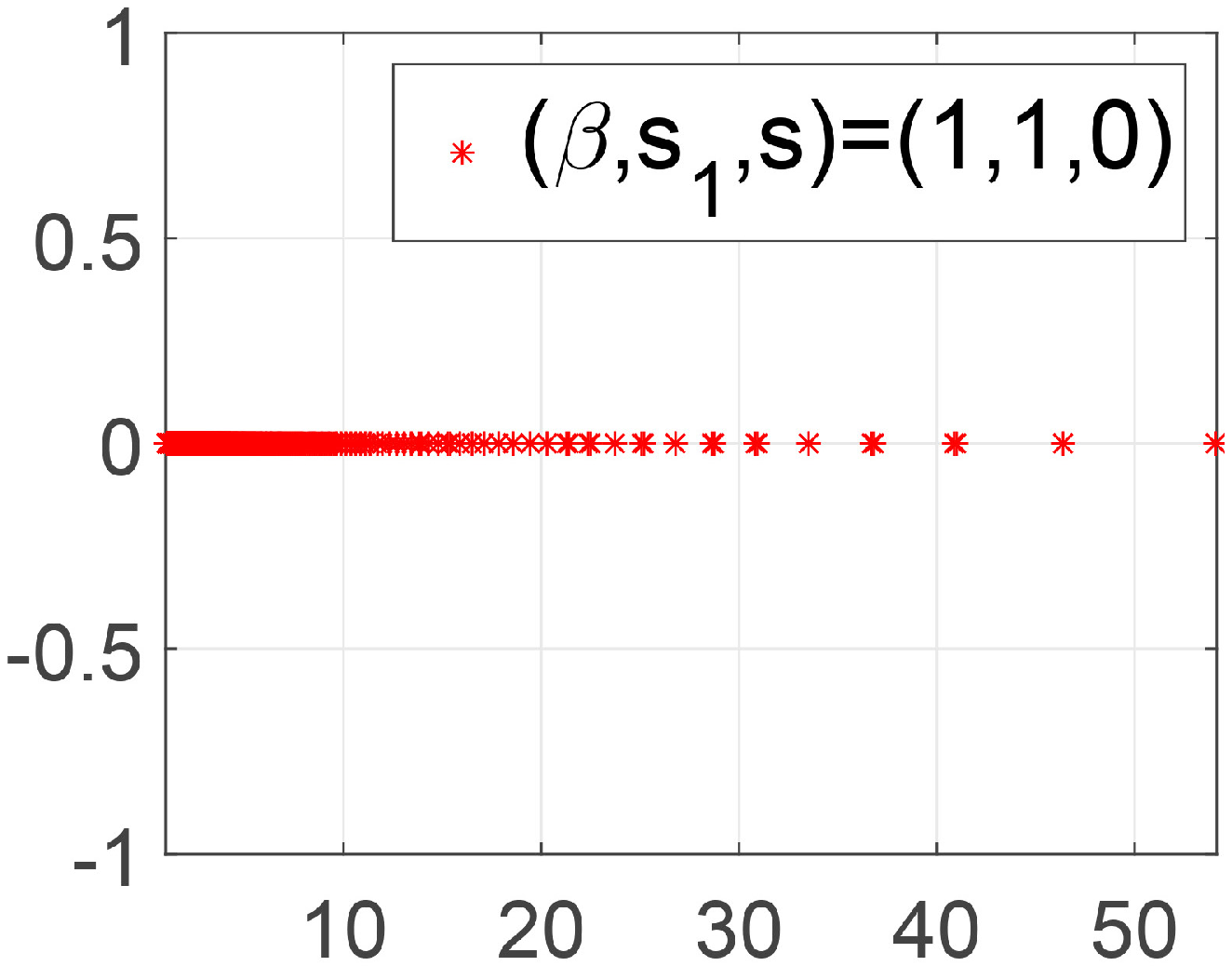}
\includegraphics[width=1.45in,height=1.2in,angle=0]{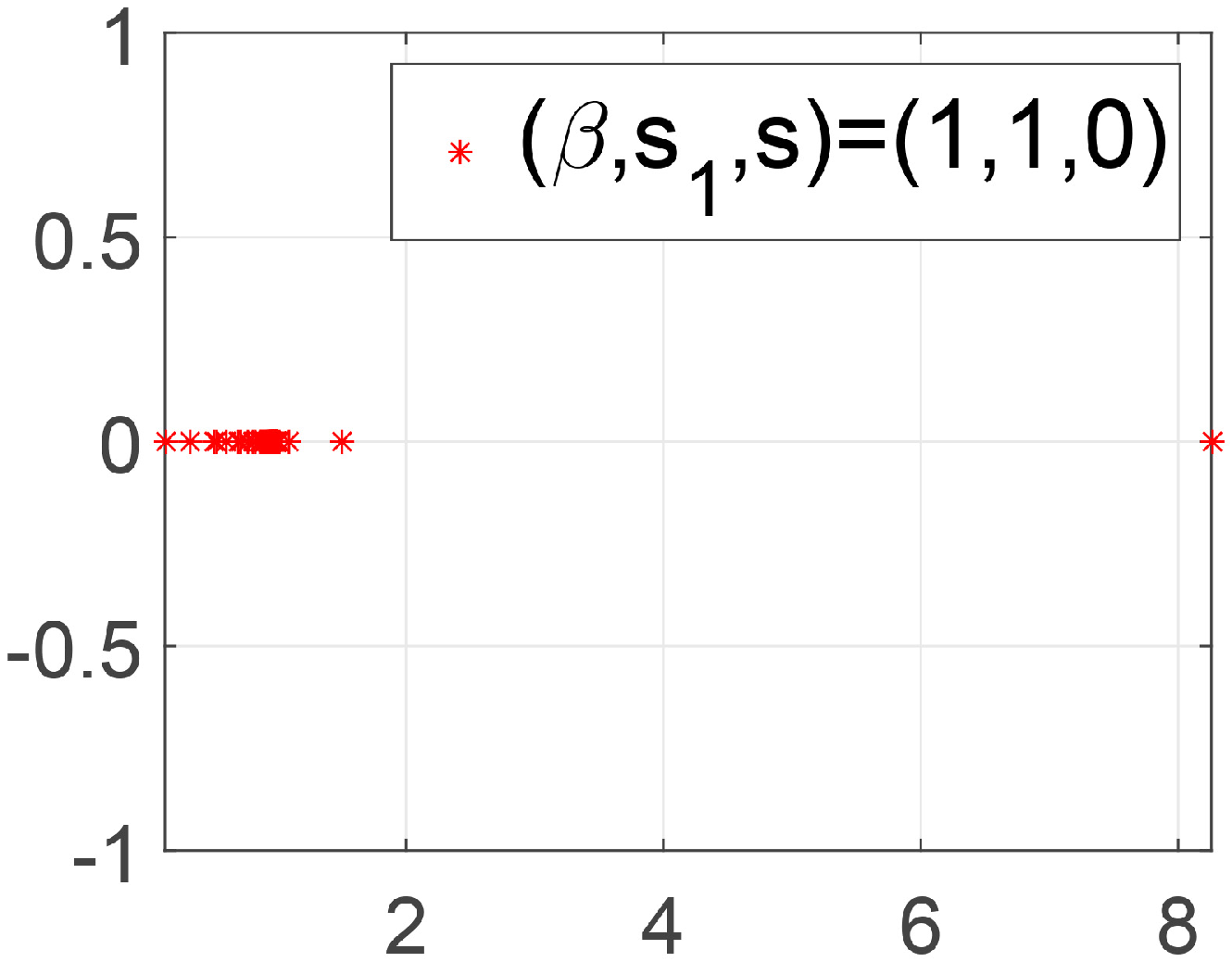}\\
\includegraphics[width=1.45in,height=1.2in,angle=0]{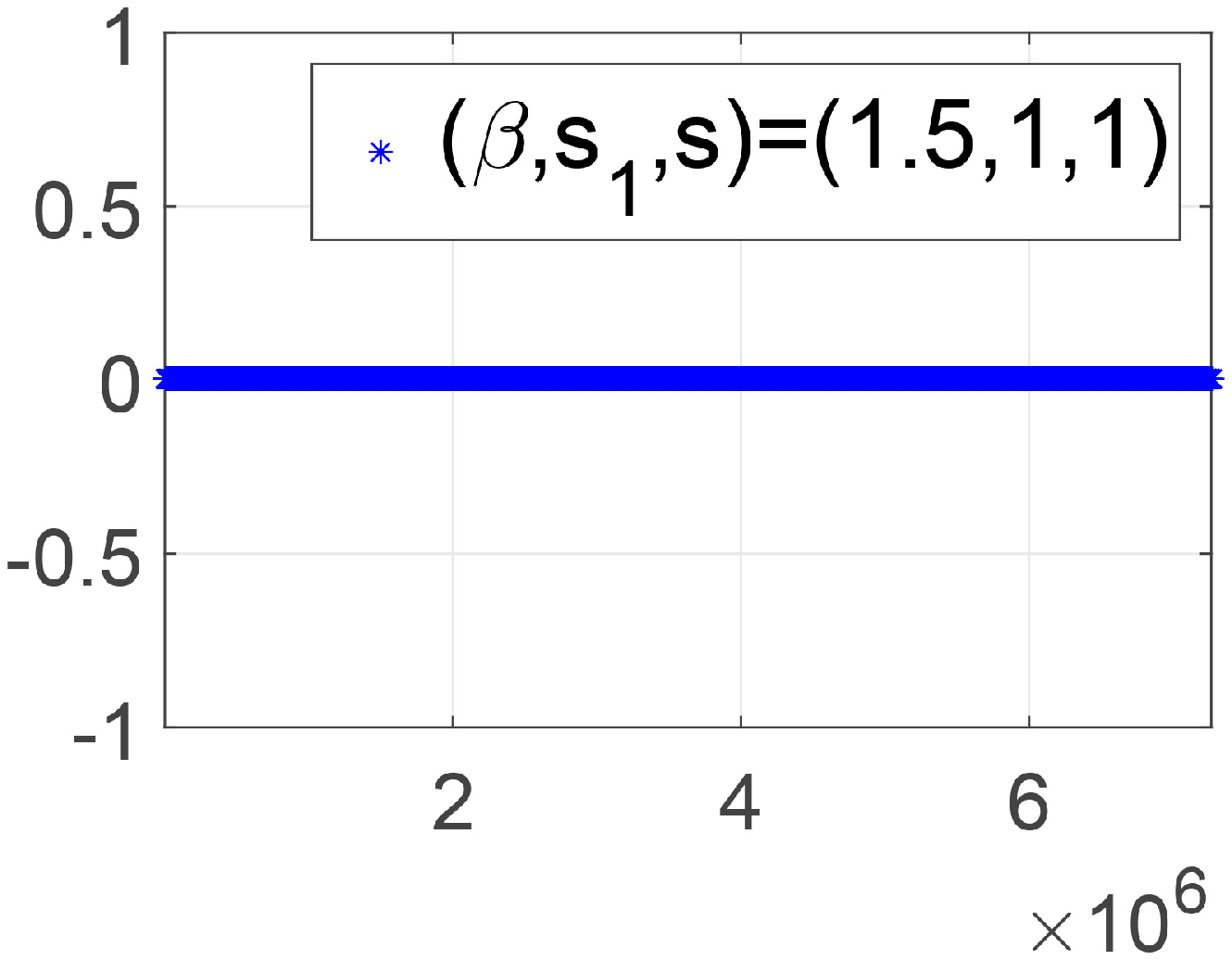}
\includegraphics[width=1.45in,height=1.2in,angle=0]{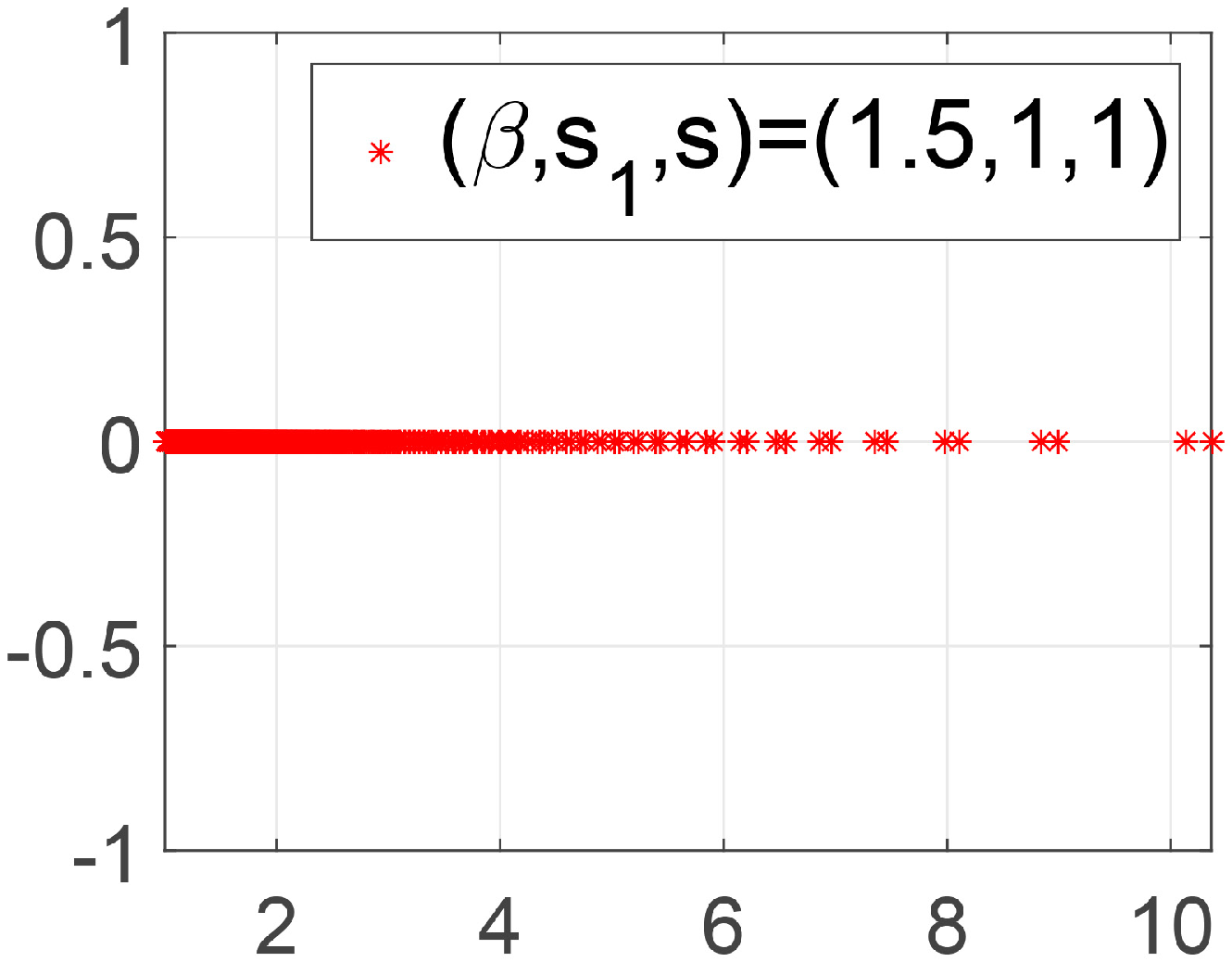}
\includegraphics[width=1.45in,height=1.2in,angle=0]{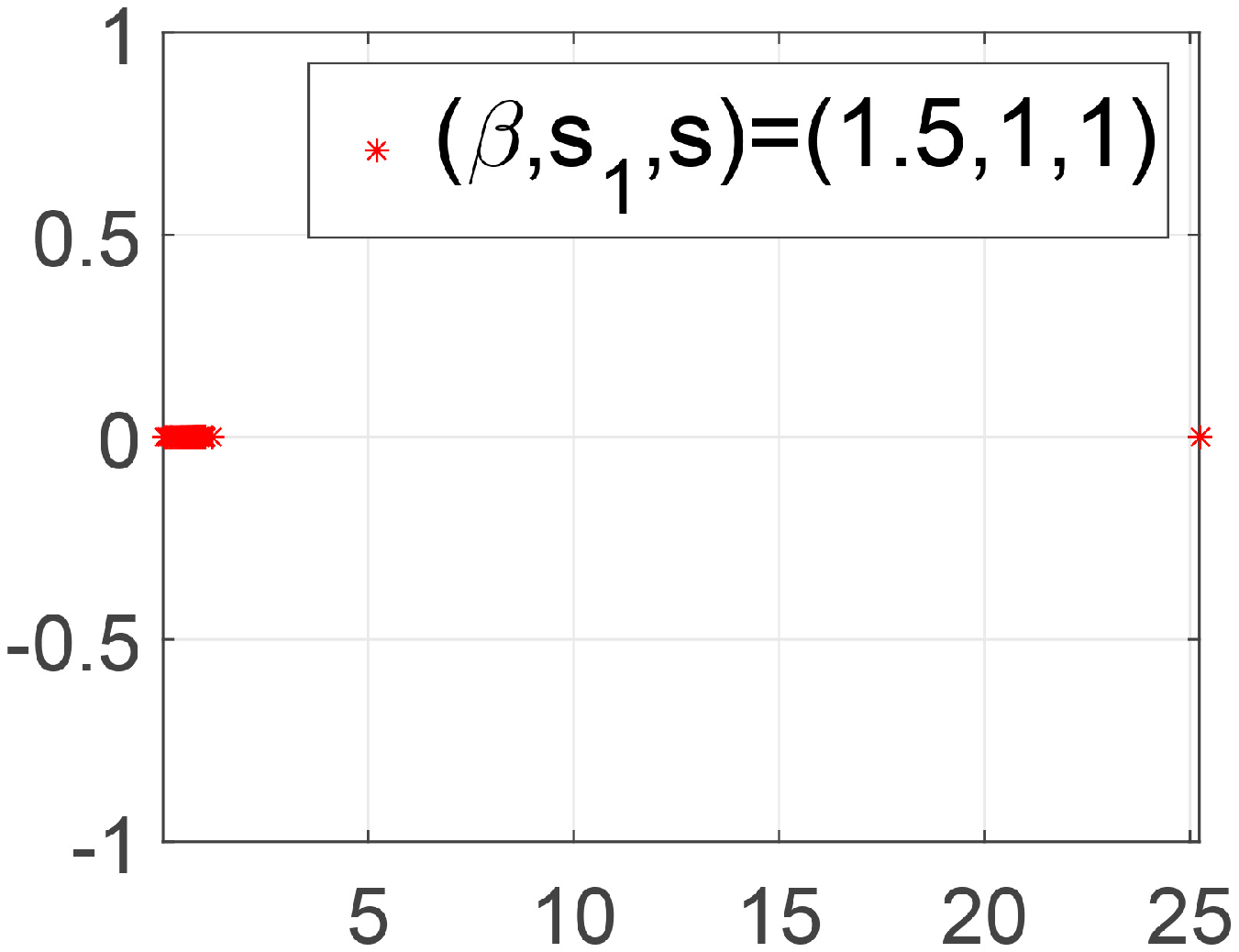}
\caption{Eigenvalue distribution of the systems (\ref{matrixform}) or (\ref{preconditionsystem}) in Example \ref{example1} with $J=13$ and $\lambda=3$.
 First column: without preconditioning; second column: preconditioned with the ichol factorization; third column: preconditioned with the  T. Chan circulant matrix.  The horizontal and vertical axes are respectively the real and imaginary axis.
 }\label{figure:1-1}
\end{center}
\end{figure}

\begin{example}\label{example2}
Consider model (\ref{Laplacedirichelt}) in  $\Omega=(0,1)$ with the boundary condition
\begin{eqnarray}
g(x)=(-2x)\chi_{[-\frac{1}{2},0]}+(2x-2)\chi_{[1,\frac{3}{2}]}
\end{eqnarray}
and    source term $f(x)$ being derived from the exact solution
\begin{eqnarray}
p(x)=(-2x)\chi_{[-\frac{1}{2},0]}+(x-x^2)^2 \chi_{(0,1)}+(2x-2)\chi_{[1,\frac{3}{2}]}.
\end{eqnarray}
\end{example}
Obviously, $u(x)$  is  discontinuous at $x=\frac{3}{2}$ and $x=-\frac{1}{2}$.  In the  numerical simulation,
for $\lambda=0$, the $f(x_i), \, d_1(i),\, d_2(i)$ are obtained exactly; for $\lambda \not=0$, they  are  calculated numerically
 with the techniques as in Example \ref{example1}.
 The numerical results are listed in  Table \ref{table:2-1}.
 Since  $u(x)$ is smooth enough
on $\bar{\Omega}$, the convergence rates are consistent with the theoretical predictions in Theorems \ref{theoremsection11} and \ref{theoremsection21}. In fact,   for $\lambda=0$ and $\beta\in (1,2)$, the numerical schemes obtained with $s_1=s=1$ seem to have a slightly bigger  convergence rate than $3-\beta$.

\begin{table}[!h t b p]\fontsize{6.8pt}{10pt}\selectfont
\begin{center}
 \caption{Errors and convergence rates  in Example \ref{example2} with $M=2^J$, solved by the PCG method with T. Chan's preconditioner.}
\begin{tabular}{cc|ccccc|ccccc}
  \hline
 $(\beta,\,s,\,s_1)$& $J$   &\multicolumn{5}{c|}{ $\lambda=0$ }    &\multicolumn{5}{|c}{$\lambda=3$  }\\
     &           & $L^2$-Err      & Rate      &$L^{\infty}$-Err  & rate  &iter   &$L^2$-Err  & Rate     &$L^\infty$-Err  & Rate &iter\\
   \hline

                  &$11$      &  1.1146e-06  &  --       & 1.6944e-06  & --   &8 &   5.9414e-06    &--       &  9.6168e-06   &--&9\\
   $(0.5,\,0,\,0)$    &$12$     &3.9593e-07  & 1.49    &6.0100e-07      & 1.50 & 8&2.1173e-06   & 1.49     &3.4273e-06    &1.49&9\\
                   &$13$     &1.4043e-07   &  1.50    & 2.1297e-07   &  1.50 & 8 &  7.5274e-07     &1.49      &1.2185e-06   &1.49 &10\\

 \\[-5pt]
                     &$11$     &   4.5687e-09   & --   & 6.7988e-09    &  --    & 8 & 4.8132e-08     & --      & 7.1820e-08  &--&9 \\
   $(0.5,\,1,\,1)$     &$12$     & 1.1421e-09   &  2,00    &1.6999e-09     & 2.01  &8   &1.2059e-08  &   2.00      &1.7999e-08    &2.00 &9\\
                    &$13$     & 2.8551e-10 &  2.00   &  4.2566e-10      & 2.01  & 8&3.0193e-09    &  2.00     & 4.5084e-09   &2.00&9\\

 \hline

                    &$11$     &  3.7770e-06   &  --   & 6.2802e-06    &  --  &12&  4.5179e-05    & --        &  6.3557e-05  &--&13 \\
$(1.0,\,1,\,0)$        &$12$     &1.8917e-06  &  1.00     &3.1450e-06  &1.00  & 12 &  22.2684e-05   &0.99       &  3.1930e-05   &0.99&13\\
                    &$13$     & 9.4664e-07   & 1.00    & 1.5737e-06 & 1.00 &13    &1.1368e-05     &  1.00     &1.6007e-05    &1.00&14\\

 \\[-5pt]
                  &$11$     & 9.4496e-09  & --    & 51.3552e-08     & --   &12&1.6933e-07     & --    & 2.6506e-07   &--&13\\
                  &$12$     &  2.3615e-09 &  2.13   &3.3872e-09     &2.13   & 12  &4.5438e-08  &2.02        & 7.1225e-08     &2.02&13\\
$(1.0,\,1,1)$     &$13$      &5.9040e-10  & 2.12    &8.4681e-10    & 2.12&  13   &1.2138e-08    & 2.02    & 1.9048e-08   &2.02&14\\

 \hline

               &$11$  & 5.5961e-05      & --    & 9.0153e-05     & --  & 19  &9.7334e-05    & --    & 1.5346e-04     &--&20\\
                 &$12$   & 3.9590e-05  & 0.50   & 6.3783e-05 & 0.50   &  21 & 6.9986e-05   & 0.48      &  1.1035e-04     &0.48&22\\
 $(1.5,\,1,\,0)$ &$13$   &2.8003e-05   & 0.50   & 4.5116e-05     &0.50   &24 &4.9910e-05     &0.49      & 7.8692e-05     &0.49 &25\\

 \\[-5pt]
                     &$10$  &   6.2783e-08 &  --    & 8.7275e-08     &--       & 16       &---     &--                 &--     &-- &-- \\
                    &$11$  &  1.5532e-08&  2.01    &  2.1605e-08     &2.01    & 19  & 2.3015e-06   &1.49       &3.6244e-06  &-- &20 \\
 $(1.5,\,1,\,1)$      &$12$ & 3.8558e-09   & 2.01    & 5.3661e-09    & 2.01&     21  & 8.1576e-07      & 1.50     & 1.2851e-06    &1.50 &23\\
                    &$13$  &  9.5433e-10 &  2.01    &1.3282e-09     & 2.01   & 24  &2.8888e-07     & 1.50      &4.5524e-07   &1.50 &25 \\

                       \hline
\end{tabular} \label{table:2-1}
\end{center}
\end{table}

\begin{example}\label{example3}
Consider  model (\ref{Laplacedirichelt})  in $\Omega=(-r,r)$ with source term   $f(x)=1$ and absorbing boundary condition $g(x)=0$.
\end{example}
 When $\lambda=0$, the exact solution  is  \cite[Subsection 3.1]{Deng:17}
\begin{eqnarray}
u(x)=\frac{\sqrt{\pi}(r^2-x^2)^{\beta/2}}{2^{\beta}\Gamma(1+\beta/2)\Gamma(1/2+\beta/2)}~~~{\rm for}~x\in \Omega.
\end{eqnarray}
It is easy to see that $u(x)$  has a poor regularity at the boundaries of $\Omega$.
When $\lambda\not=0$, $u(x)$ cannot be obtained explicitly; the errors (i.e., the data under  $\tilde{L}^2$-Err and ${\tilde{L}}^\infty$-Err) under stepsize $h$ in Table \ref{table:3-1}  with the $L^2$ and $L^{\infty}$ norms are, respectively,
\begin{eqnarray}\label{example3errordef}
\left\|U_{h/2}-U_h\right\|~~{\rm and}~~\max_{1\le i\le M}\left|U_{h/2}-U_h\right|,
\end{eqnarray}
 and the convergence rates are measured by using there errors, where $h=\frac{2r}{M+1}$ with $M=2^J-1$.
The  numerical results in Table \ref{table:3-1} show that  the convergence rates are small for two  cases, being consistent with the results in \cite[Example 2]{Zhang:17}.

 In statistical physics, the solution $u(x)$ denotes the mean first exit time of a particle  starting at $x$ away
 from the given domain $\Omega$ \cite{Deng:16,Getoor:61}.
For $\lambda=0,\, 0.5$ and  $3$, the numerical solutions obtained  with $(s,s_1)=(1,1)$ and different values of $\beta=0.5, \,1,\,1.5, \, r=1, \,2, \,5$ are listed in Figures \ref{figure:3-1}, which show: for the same domain $\Omega$, the mean first exit time  increases with the increases of the value of $\lambda$; and when $\lambda>0$, for any fixed value of the starting point $x$, the mean exit times are shorter for larger values of $\beta$.

\begin{table}[!htbp]\fontsize{6.5pt}{10pt}\selectfont
\begin{center}
 \caption{Errors and convergence rates  in Example \ref{example3} with $M=2^J-1$, solved by the PCG method with  T. Chan's preconditioner.
 The data in the parentheses denote the rates obtained with the errors in (\ref{example3errordef}).}
\begin{tabular}{cc|cccc|cccc}
  \hline
     &      &\multicolumn{4}{c|}{$\lambda=0$ }    &\multicolumn{4}{|c}{$\lambda=3$  }\\

  $(\beta,\,s,\,s_1)$  &$M$    & $L^2$-Err  & Rates     &${L}^\infty$-Err  & rates    & $\tilde{L}^2$-Err  & Rate     &${\tilde{L}}^\infty$-Err  & rate   \\
   \hline%
                     &$11$     &3.5668e-03   & --             & 5.4682e-02      & --            &3.3886e-02       & --    & 2.2117e-01     &-- \\
    $(0.5,\,0,\,0)$    &$12$     &2.1205e-03  &  0.75\,(0.75)   & 4.5975e-02    &   0.25\,(0.25)   &1.6834e-02      &1.01   &1.8113e-01       &0.29 \\
                     &$13$     &1.2608e-03   &  0.75\,(0.75)     & 3.8658e-02      &  0.25\,(0.25)    &8.5734e-03    &0.97    &1.4951e-01     &0.28 \\


                    &$11$     &1.4675e-03    & --            &  2.0163e-02     & --            &1.3158e-02     & --    &8.0461e-02     &-- \\
   $(0.5,\,1,\,1)$  &$12$     &8.7272e-04   &  0.75\,(0.75)  & 1.6952e-02      & 0.25\,(0.25)   &6.9321e-03     & 0.92 & 6.5320e-02     &0.30 \\
                  &$13$     & 5.1896e-04    & 0.75\,(0.75)   & 1.4254e-02      & 0.25\,(0.25)   &3.6595e-03    &0.92    &5.3565e-02     &0.29 \\

  \hline
                     &$11$     &6.3631e-04   & --              & 4.7004e-03      & --            & 1.6940e-03     & --    &  6.1020e-03     &-- \\
    $(1.0,\,1,\,0)$    &$12$     &3.3183e-04   &  0.94\,(0.94)   &3.3232e-03     & 0.50\,(0.50)   &8.7755e-04   &0.95  &4.2848e-03      &0.51 \\
                     &$13$     &1.7249e-04  &   0.94\,(0.94)  & 2.3497e-03    &   0.50\,(0.50)    &4.4946e-04    & 0.97    &3.0161e-03     &0.51 \\

                       \\[-5Pt]

                    &$11$     & 6.3631e-04   & --               &  4.7004e-03    & --        &1.6940e-03     & --    & 6.1020e-03     &-- \\
   $(1.0,\,1,\,1)$  &$12$     &3.3183e-04     &  1.08\,(1.08)  &  3.3232e-03  & 0.64\,(0.64)   &8.7755e-04     & 1.09  &  4.2848e-03    &0.65 \\
                  &$13$     &1.7249e-04   &  1.07\,(1.07)  & 2.3497e-03    &  0.63\,(0.63)   &4.4946e-04    & 1.09    &3.0161e-03     &0.63 \\
     \hline
                    & $11$     &   1.6135e-03  & --   & 1.3157e-03             & --              &2.7348e-03      & --    &  2.4835e-03     &-- \\
    $(1.5\,\,1,\,0)$    &$12$     &1.1034e-03  & 0.55\,(0.58)   & 9.0676e-04  &  (0.54)\, (0.56)  &1.9986e-03    &0.45   & 1.8304e-03      &0.44 \\
                     &$13$     &7.6158e-04  &   0.53\,(0.56)  &6.2939e-04     &  (0.53)\, (0.54)   &1.4278e-03   & 0.49    &1.3146e-03    &0.48 \\

                       \\[-5pt]

                    &$11$     & 2.2426e-04   & --    &  5.4375e-04      & --    &1.8108e-04      & --    & 5.7640e-04     &-- \\
   $(1.5,\,1,\,1)$  &$12$     &1.1260e-04      &  0.99\,(0.99) & 3.2327e-04     & 0.75\,(0.75)   &1.1576e-04      & 0.65  & 3.4404e-04     &0.74 \\
                  &$13$     &5.6465e-05  &  1.00\,(1.00)   & 1.9220e-04     &  0.75\,(0.75)  &6.9126e-05     & 0.74    &2.0489e-04     &0.75 \\

                    \hline

\end{tabular}\label{table:3-1}
\end{center}
\end{table}

\begin{figure}[!htbp]
\centering
\includegraphics[width=1.6in,height=1.6in,angle=0]{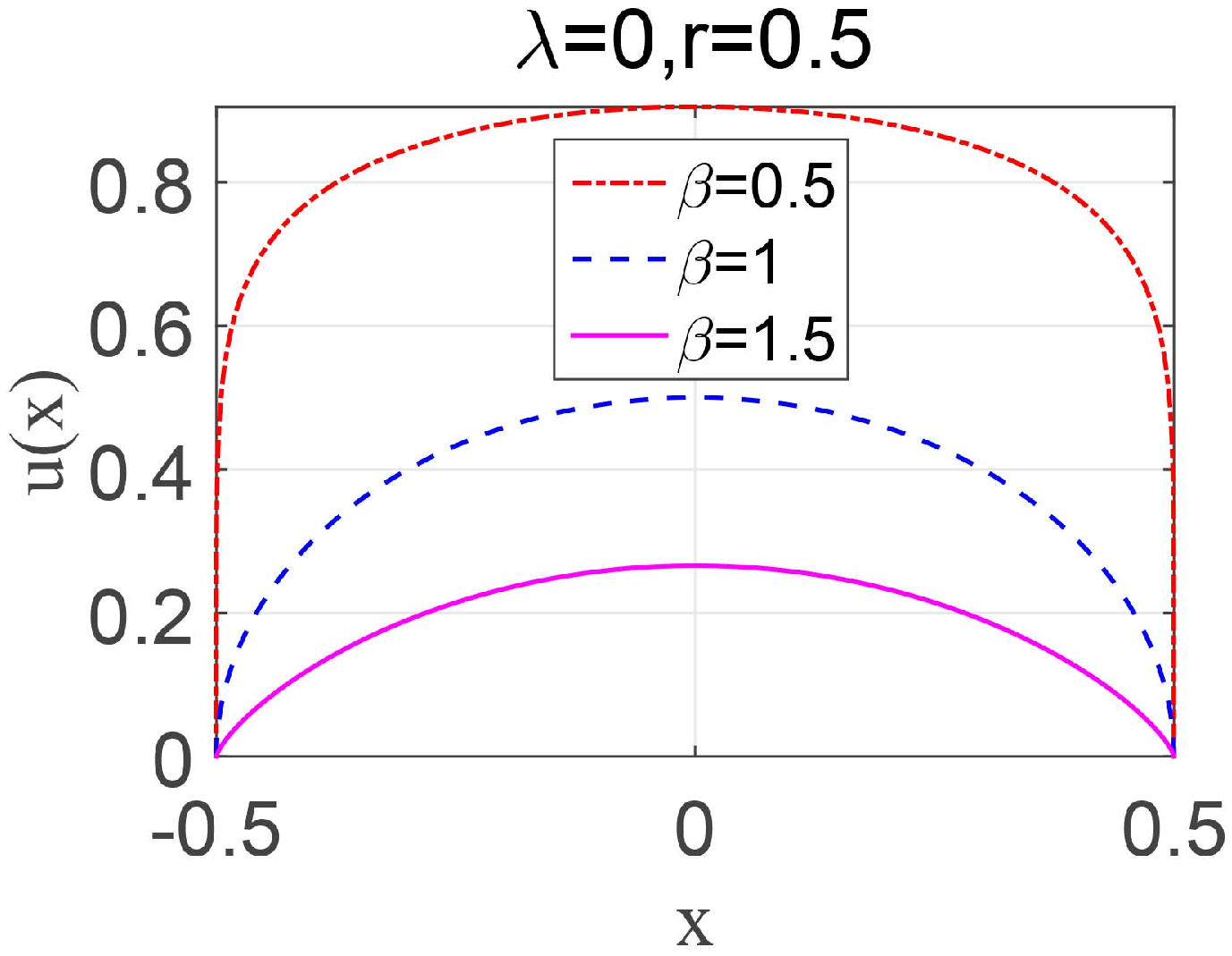}
\includegraphics[width=1.6in,height=1.6in,angle=0]{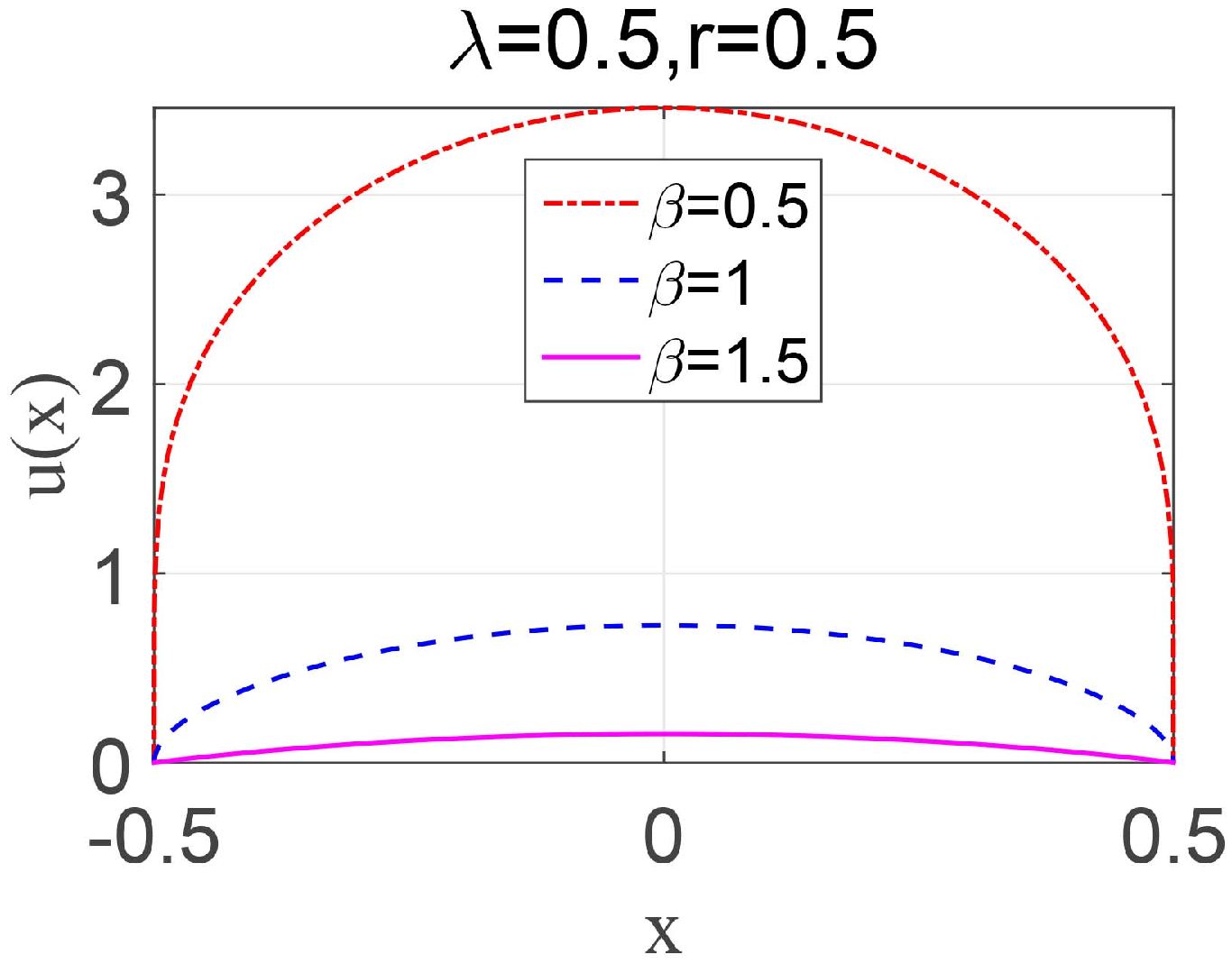}
\includegraphics[width=1.6in,height=1.6in,angle=0]{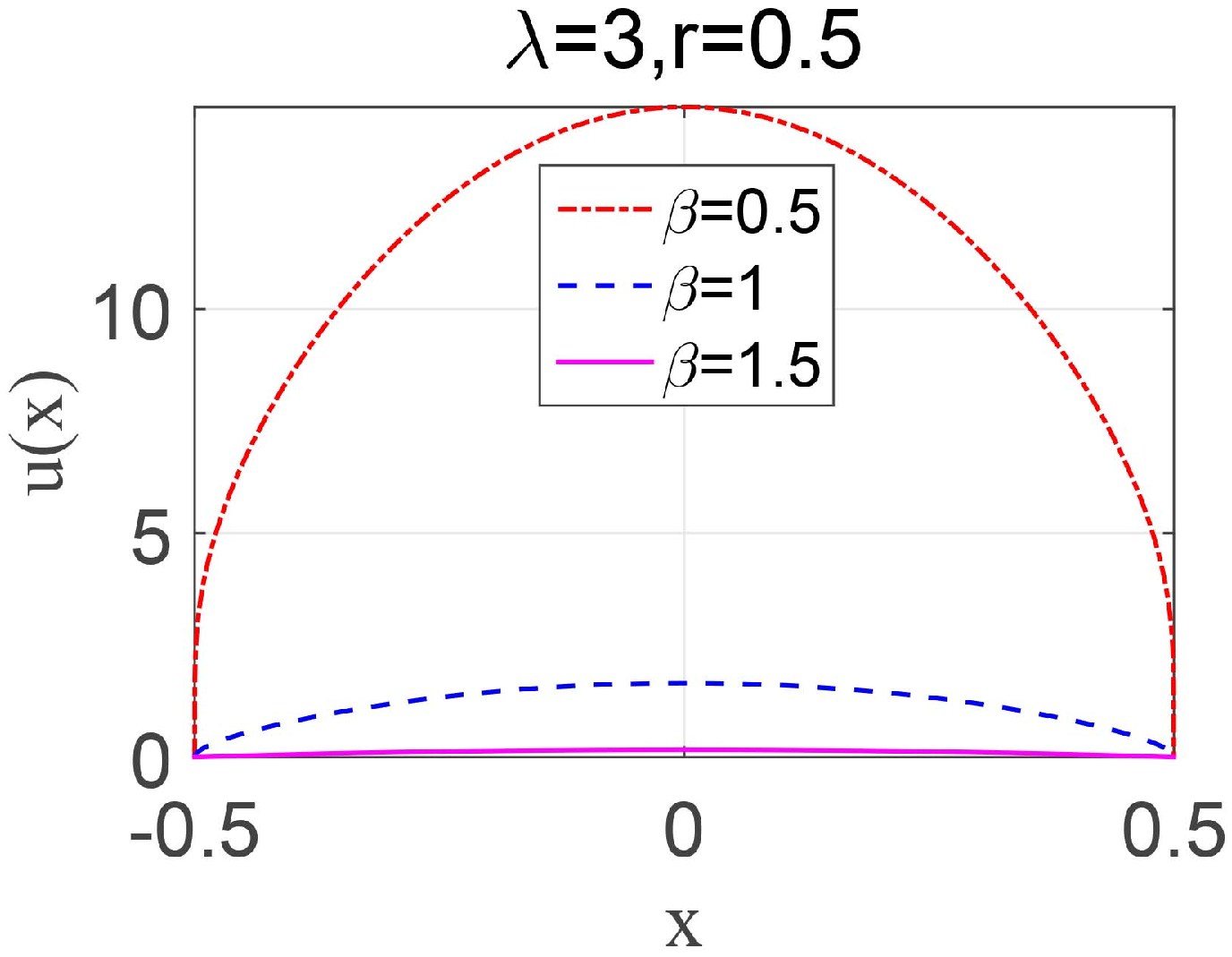}\\
\includegraphics[width=1.6in,height=1.6in,angle=0]{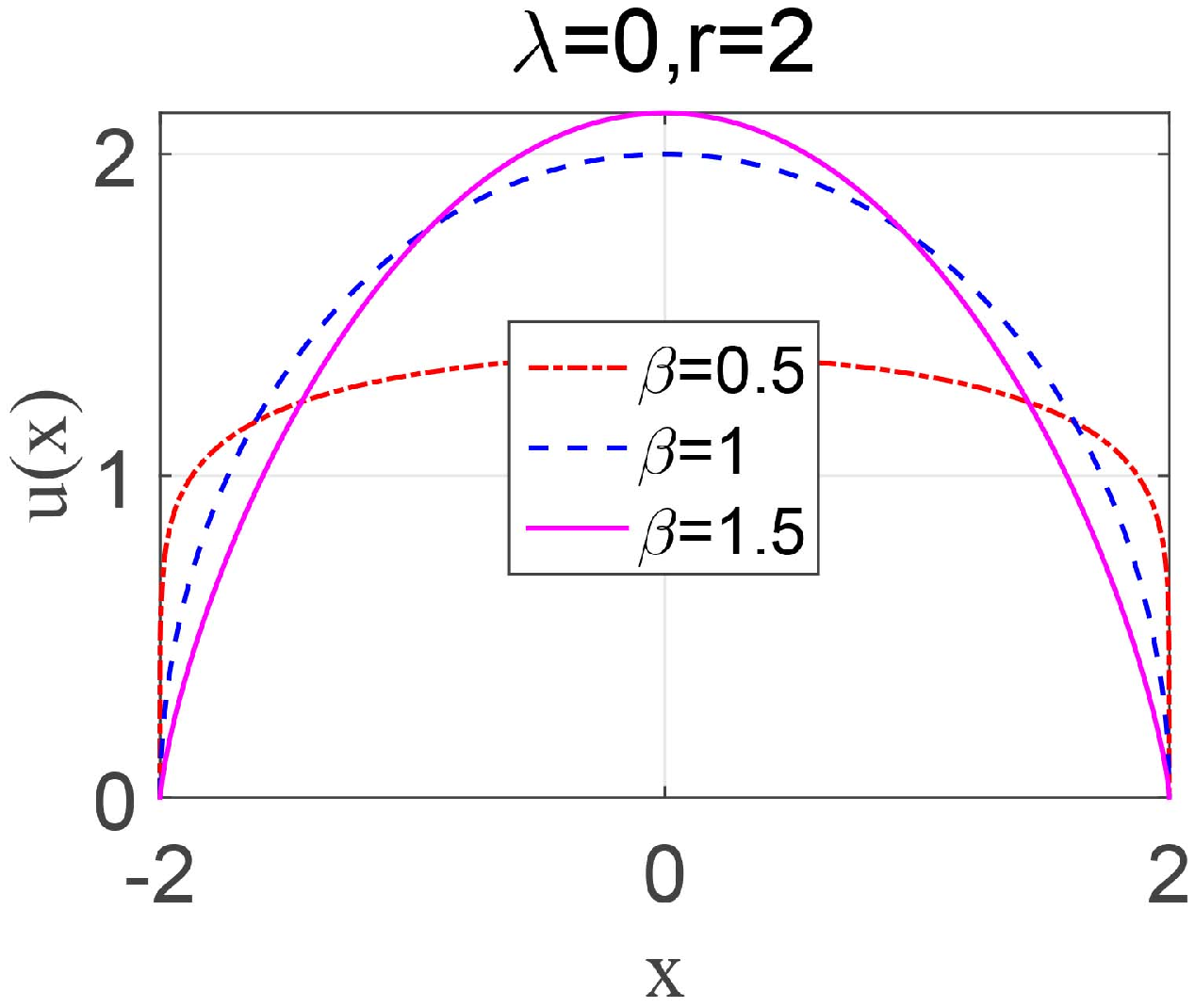}
\includegraphics[width=1.6in,height=1.6in,angle=0]{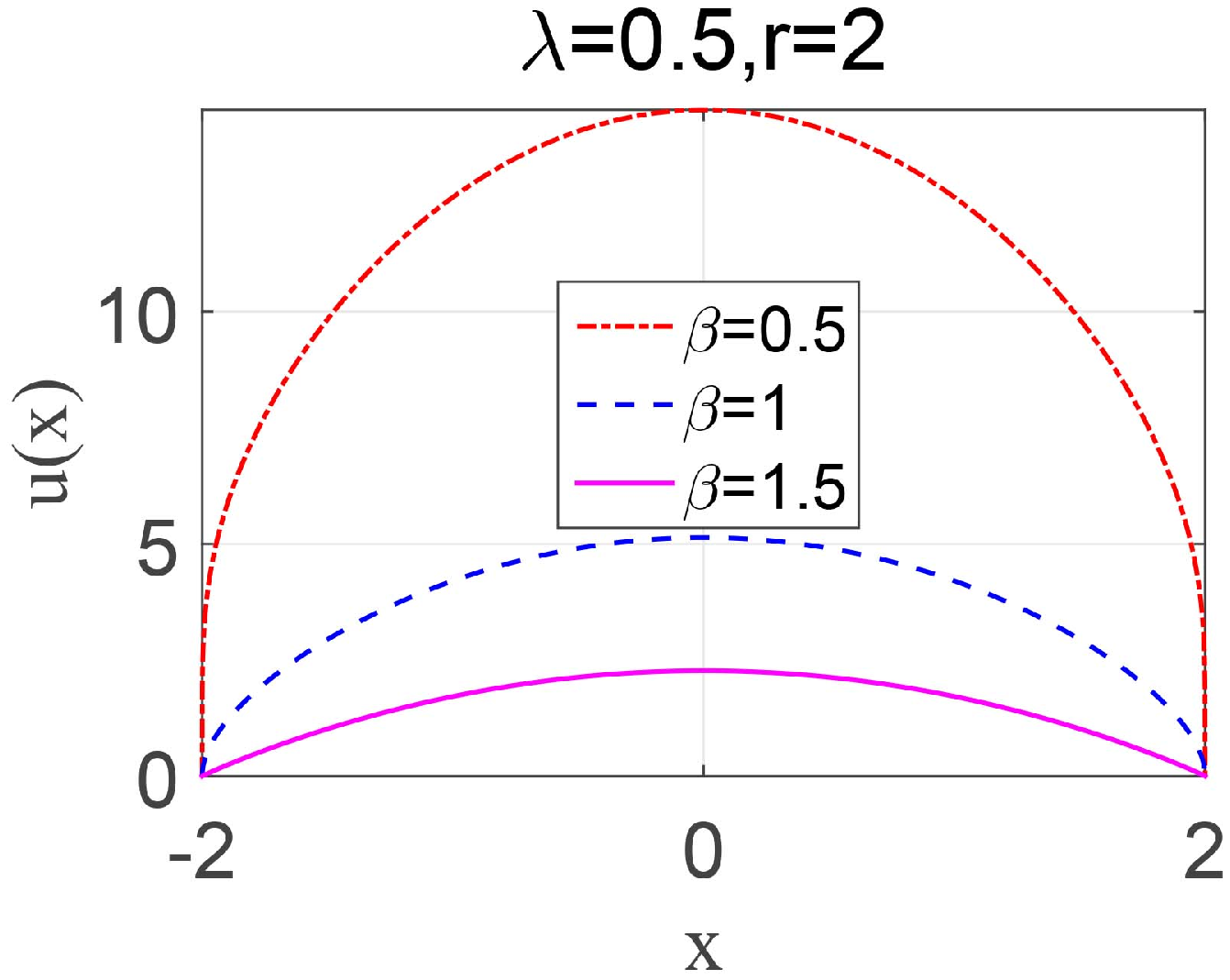}
\includegraphics[width=1.6in,height=1.6in,angle=0]{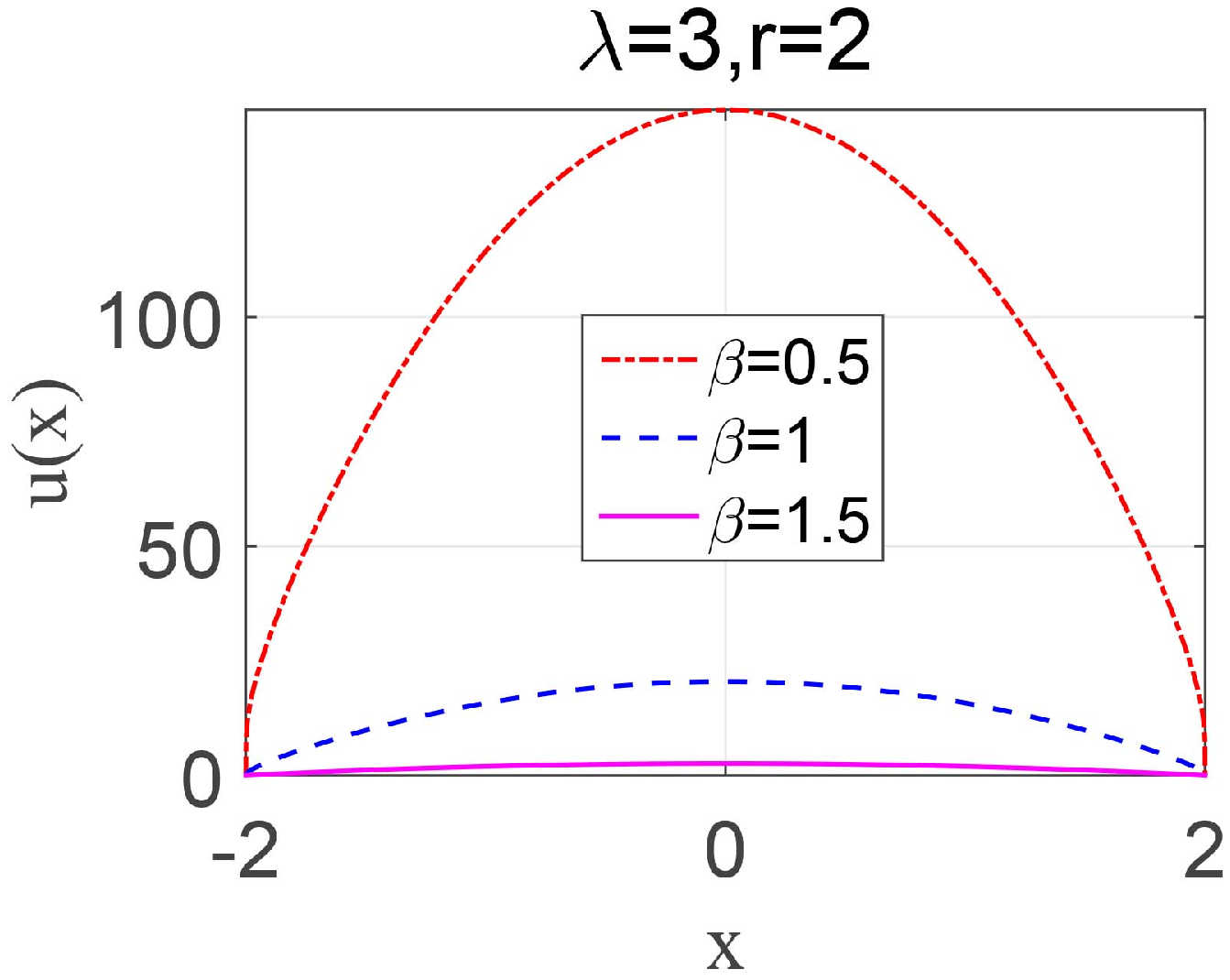}\\
\includegraphics[width=1.6in,height=1.6in,angle=0]{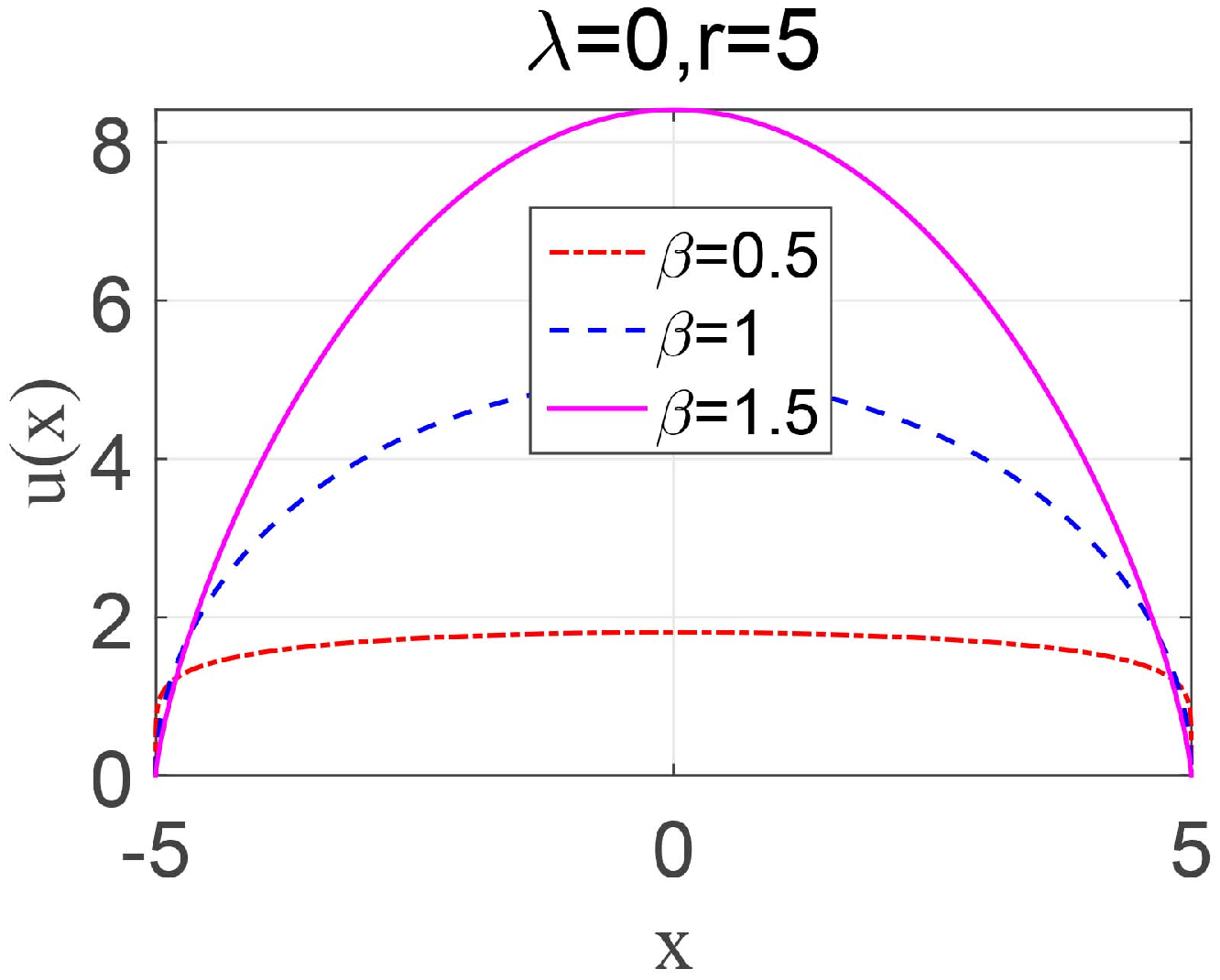}
\includegraphics[width=1.6in,height=1.6in,angle=0]{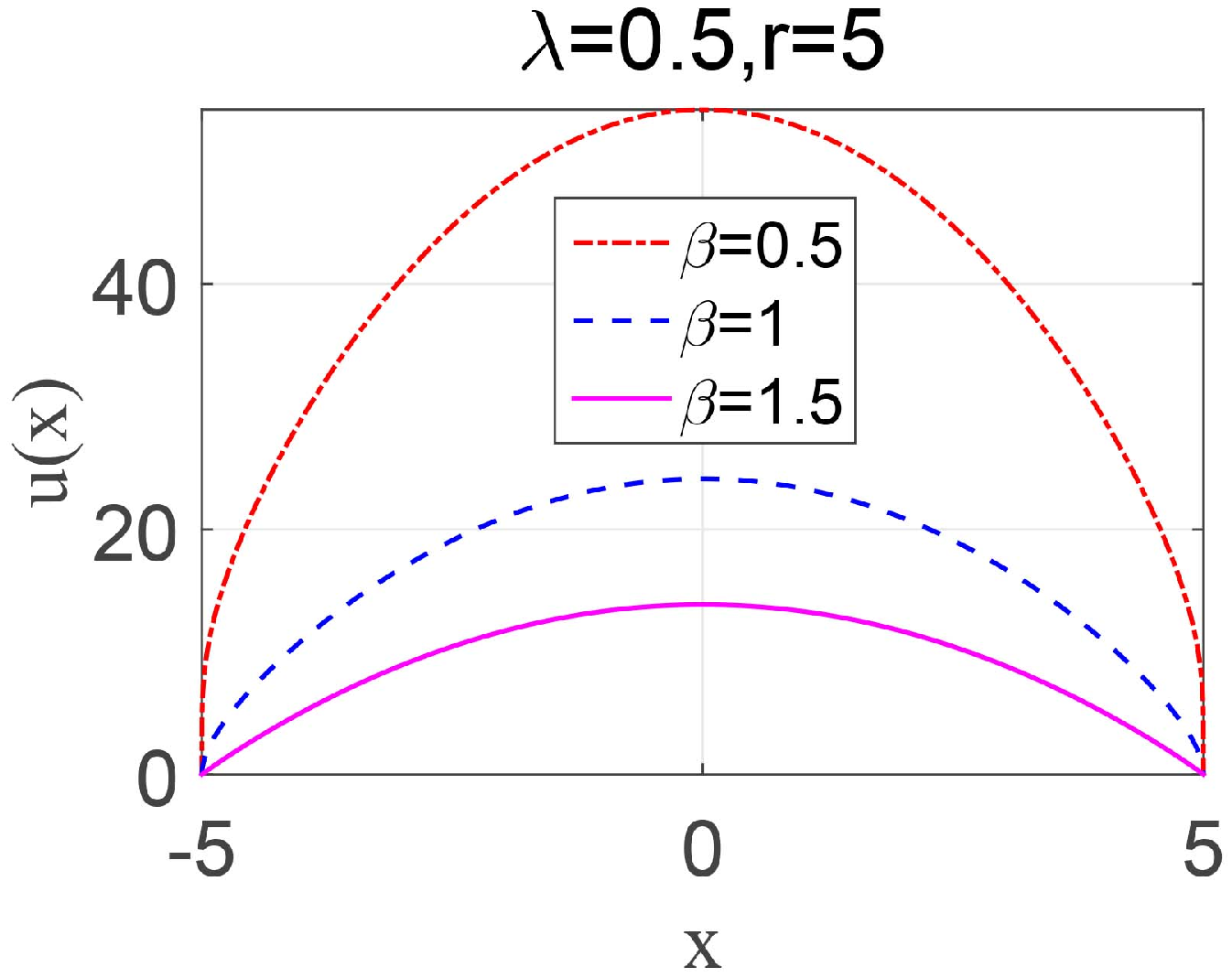}
\includegraphics[width=1.6in,height=1.6in,angle=0]{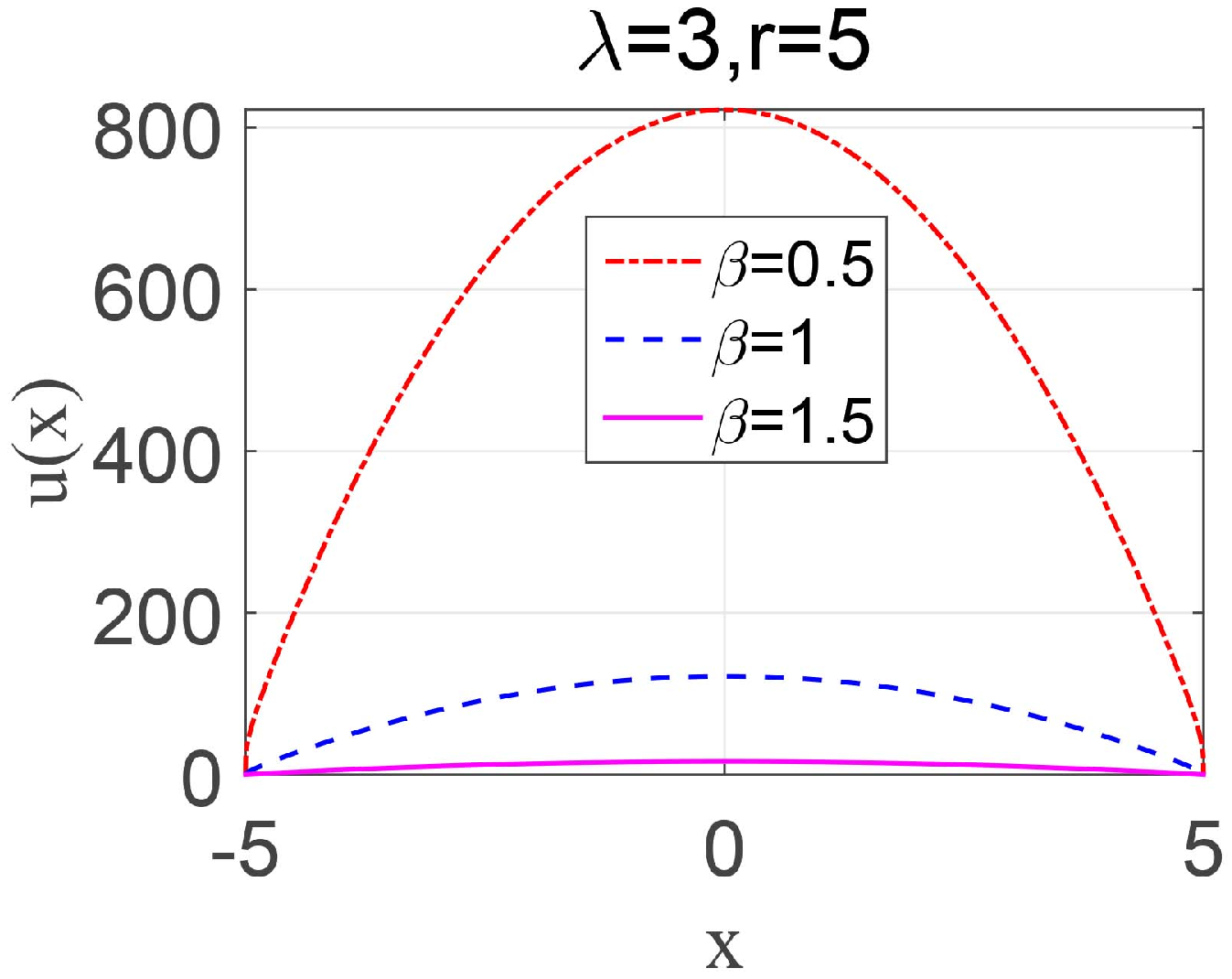}
\caption{Dependence of the mean exit time $u(x)$ on the size of domain $\Omega$, and the values of $\beta$ and $\lambda$.}
\label{figure:3-1}
\end{figure}

\section{Conclusion}\label{sec5}
The tempered fractional Laplacian is a recently introduced operator for treating the weaknesses of the fractional Laplacian in some physical processes. This paper provides several classes of finite difference schemes for the tempered fractional Laplacian equation with $\beta \in (0,2)$. According to the regularity of the solution, one can choose the more appropriate numerical schemes. The detailed numerical analyses are performed, and the effective preconditioning techniques are provided. The implementation details are also discussed, including that the  entries of the stiffness matrix can be explicitly and conveniently calculated and the stiffness matrix has Toeplitz-like structure. The efficiencies of the algorithms are verified by extensive numerical experiments and the desired convergence rates are confirmed.


\section*{Acknowledgments}
This work was supported by the National Natural Science Foundation of China under Grant Nos. 11671182 and 11701456, and the Fundamental Research Funds for the Central Universities under Grant No. lzujbky-2017-ot10.


\end{document}